\PassOptionsToPackage{dvipsnames}{xcolor}
\documentclass[twoside]{article}
\usepackage[T1]{fontenc}
\usepackage{geometry}
\usepackage{amsmath,amsthm,amsfonts,amssymb}
\usepackage{bm}
\usepackage{todonotes}
\usepackage{proof}
\usepackage{authblk}
\usepackage{mathtools}
\usepackage{tikz,tikz-cd,tikzpeople}
\usepackage{tikzit}				

\tikzstyle{morphism}=[fill=white, draw=black, shape=rectangle]
\tikzstyle{medium box}=[fill=white, draw=black, shape=rectangle, minimum width=0.7cm, minimum height=0.7cm]
\tikzstyle{large morphism}=[fill=white, draw=black, shape=rectangle, minimum width=1.7cm, minimum height=1cm]
\tikzstyle{bn}=[fill=black, draw=black, shape=circle, inner sep=1.5pt]
\tikzstyle{state}=[fill=white, draw=black, regular polygon, regular polygon sides=3, minimum width=0.8cm, shape border rotate=180, inner sep=0pt]
\tikzstyle{medium state}=[fill=white, draw=black, regular polygon, regular polygon sides=3, minimum width=1.3cm, inner sep=0pt, shape border rotate=180]
\tikzstyle{large state}=[fill=white, draw=black, regular polygon, regular polygon sides=3, minimum width=2.2cm, shape border rotate=180, inner sep=0pt]
\tikzstyle{wide state}=[fill=white, draw=black, shape=isosceles triangle, minimum width=0.8cm, shape border rotate=270, inner sep=1.4pt, minimum height=0.5cm, isosceles triangle apex angle=80]
\tikzstyle{blue morphism}=[fill=white, draw={rgb,255: red,15; green,0; blue,150}, shape=rectangle, text={rgb,255: red,15; green,0; blue,150}, tikzit category=blue]
\tikzstyle{blue state}=[fill=white, draw={rgb,255: red,15; green,0; blue,150}, shape=circle, regular polygon, regular polygon sides=3, minimum width=0.8cm, shape border rotate=180, inner sep=0pt, text={rgb,255: red,15; green,0; blue,150}, tikzit category=blue]
\tikzstyle{blue node}=[fill={rgb,255: red,15; green,0; blue,150}, draw={rgb,255: red,15; green,0; blue,150}, shape=circle, tikzit category=blue, inner sep=1.5pt]
\tikzstyle{blue}=[text={rgb,255: red,15; green,0; blue,150}, tikzit draw={rgb,255: red,191; green,191; blue,191}, tikzit category=blue, tikzit fill=white, inner sep=0mm]
\tikzstyle{blue wide state}=[fill=white, draw={rgb,255: red,15; green,0; blue,150}, text={rgb,255: red,15; green,0; blue,150}, shape=isosceles triangle, minimum width=0.8cm, shape border rotate=270, inner sep=1.4pt, minimum height=0.5cm, isosceles triangle apex angle=80]
\tikzstyle{white morphism}=[fill=white, draw=white, shape=rectangle, tikzit draw={rgb,255: red,139; green,139; blue,139}]
\tikzstyle{horiz state}=[fill=white, draw=black, regular polygon, regular polygon sides=3, minimum width=1cm, shape border rotate=90, inner sep=0pt]
\tikzstyle{arrow}=[->]
\tikzstyle{blue arrow}=[-, draw={rgb,255: red,15; green,0; blue,150}, tikzit category=blue]
\tikzstyle{dashed box}=[-, dashed]
\tikzstyle{mapsto}=[{|->}]
\tikzstyle{protected}=[-, preaction={line width=1.8pt,white,draw}]
\tikzstyle{double wire}=[-, draw, line width=0.8pt, white, preaction={-, draw, line width=1.6pt}]
\tikzstyle{protected double wire}=[-, draw, line width=0.8pt, white, preaction={-, draw, line width=1.6pt, preaction={line width=3pt,white,draw}}]
\tikzstyle{triple wire}=[-, draw, line width=0.4pt, preaction={-, draw, line width=2pt, white, preaction={-, draw, line width=2.8pt}}]
\tikzstyle{blue double arrow}=[-, draw, line width=0.8pt, white, preaction={-, draw={rgb,255: red,15; green,0; blue,150}, tikzit category=blue, line width=1.6pt}]
\tikzstyle{d-wire1 plate}=[-, draw={red!20!white}, line width=0.8pt, preaction={-, line width=1.6pt, draw}]
\tikzstyle{d-wire2 plate}=[-,draw={red!32!white}, line width=0.8pt, preaction={-, line width=1.6pt, draw}]
\tikzstyle{dotted_plate}=[-, densely dotted, draw=red, fill opacity=0.4, fill={red!50!white}, rounded corners, tikzit fill={rgb,255: red,255; green,194; blue,195}]
\tikzstyle{twire1}=[-, draw, line width=0.4pt, preaction={-, draw, line width=2pt, red!20!white, preaction={-, draw, line width=2.8pt}}]
\tikzstyle{twire2}=[-, draw, line width=0.4pt, preaction={-, draw, line width=2pt, red!32!white, preaction={-, draw, line width=2.8pt}}]
\tikzstyle{purple line}=[-, draw={rgb,255: red,120; green,0; blue,120}]
\tikzstyle{orange line}=[-, draw={rgb,255: red,255; green,100; blue,0}]
\tikzstyle{red line}=[-, draw={rgb,255: red,150; green,0; blue,2}]
\tikzstyle{blue line}=[-, draw={rgb,255: red,15; green,0; blue,150}]
\tikzstyle{red double arrow}=[-, draw, line width=0.8pt, white, preaction={-, line width=1.6pt, draw={rgb,255: red,150; green,0; blue,2}, tikzit category=red}]
\tikzstyle{blue double arrow}=[-, draw, line width=0.8pt, white, preaction={-, draw={rgb,255: red,15; green,0; blue,150}, tikzit category=blue, line width=1.6pt}]
\tikzstyle{purple double line}=[-, draw, line width=0.8pt, white, preaction={-, line width=1.6pt, draw={rgb,255: red,120; green,0; blue,120}}]
\tikzstyle{orange double line}=[-, draw, line width=0.8pt, white, preaction={-, line width=1.6pt, draw={rgb,255: red,255; green,100; blue,0}}]
\tikzstyle{protected red double arrow}=[-, draw, line width=0.8pt, white, preaction={-, line width=1.6pt, draw={rgb,255: red,150; green,0; blue,2}, tikzit category=red, preaction={line width=3pt,white,draw}}]
\tikzstyle{protected blue double arrow}=[-, draw, line width=0.8pt, white, preaction={-, draw={rgb,255: red,15; green,0; blue,150}, tikzit category=blue, line width=1.6pt, preaction={line width=3pt,white,draw}}]
\tikzstyle{protected purple double line}=[-, draw, line width=0.8pt, white, preaction={-, line width=1.6pt, draw={rgb,255: red,120; green,0; blue,120}, preaction={line width=3pt,white,draw}}]
\tikzstyle{protected orange double line}=[-, draw, line width=0.8pt, white, preaction={-, line width=1.6pt, draw={rgb,255: red,255; green,100; blue,0},preaction={line width=3pt,white,draw}}]

\usetikzlibrary{external,arrows,decorations.pathreplacing,calligraphy,calc,positioning} 
\usepackage{enumitem}
	\setlist[enumerate]{label=(\roman*)}
	\setlist[enumerate,2]{label=(\alph*)}
\usepackage{cite}
\usepackage{xparse}				
\usepackage{caption}
\usepackage{subcaption}
\usepackage{cancel}
\usepackage{bigints}

\definecolor{myurlcolor}{rgb}{0,0,0.3}
\definecolor{mycitecolor}{rgb}{0,0.3,0}
\definecolor{myrefcolor}{rgb}{0.3,0,0}
\usepackage[pagebackref,draft=false]{hyperref}
\hypersetup{colorlinks,
linkcolor=myrefcolor,
citecolor=mycitecolor,
urlcolor=myurlcolor}

\usepackage[capitalize,noabbrev]{cleveref}
\newcommand{\sref}[2]{\hyperref[#2]{#1~\ref{#2}}}	

\newtheorem*{theorem*}{Theorem}
\newtheorem{theorem}{Theorem}[section]
\newtheorem{proposition}[theorem]{Proposition}
\newtheorem{lemma}[theorem]{Lemma}
\newtheorem{corollary}[theorem]{Corollary}
\newtheorem{definition}[theorem]{Definition}

\newtheorem{notation}[theorem]{Notation}
\theoremstyle{definition}
\newtheorem{example}[theorem]{Example}
\newtheorem{remark}[theorem]{Remark}

\newtheorem{assumption}[theorem]{Assumption}

\newcommand{\I}{\mathcal{I}}
\newcommand{\J}{\mathcal{J}}
\newcommand{\A}{\mathcal{A}}
\newcommand{\N}{\mathbb{N}}

\newcommand{\R}{\mathbb{R}}

\newcommand{\dr}{\mathrm{d}}
\newcommand{\eps}{\varepsilon}

\newcommand{\id}{\mathrm{id}} 		

\newcommand{\tensor}{\otimes}
\newcommand{\tail}{\mathrm{tail}}

\newcommand{\compl}[1]{\overline{#1}} 
\newcommand{\complsm}[1]{\widehat{#1}} 
\makeatletter 
\newcommand{\bigp@rp}[2]{%
  \vcenter{%
    \m@th\hbox{\scalebox{\ifx#1\displaystyle1.9\else1.3\fi}{$#1\perp$}}
  }%
}
\newcommand{\bigperp}{%
  \mathop{\mathpalette\bigp@rp\relax}%
  \displaylimits
}
\makeatother
\newcommand{\largeperp}{%
  \mathrel{\scalebox{1.4}{$\perp$}}%
}
\newcommand{\largemid}{%
  \mathrel{\scalebox{1.35}{$\vert$}}%
}

\newcommand{\ph}{\mathord{\rule[-0.05em]{0.6em}{0.05em}}}		

\DeclareMathOperator{\NonDesc}{\mathsf{NonDesc}}

\DeclareMathOperator{\Past}{\mathsf{Past}}
\DeclareMathOperator{\In}{\mathsf{in}}
\DeclareMathOperator{\Out}{\mathsf{out}}
\DeclareMathOperator{\timing}{\mathfrak{t}}

\newcommand{\tailcond}[1]{{#1}_{|\mathrm{tail}}}
\definecolor{parametrized}{RGB}{15,0,150}
\newcommand{\param}[1]{{\color{parametrized}#1}}

\newcommand{\newterm}[1]{\textbf{#1}}

\newcommand{\cC}{\mathsf{C}}		
\newcommand{\samp}{\mathsf{samp}}	

\newcommand{\Stoch}{\mathsf{Stoch}}
\newcommand{\BorelStoch}{\mathsf{BorelStoch}}
\newcommand{\AnaStoch}{\mathsf{AnaStoch}}
\newcommand{\UnivStoch}{\mathsf{UnivStoch}}
\newcommand{\as}[1]{
		\def\relstate{#1}%
		\ifx\relstate\empty
		  \text{a.s.}%
		\else
		  {#1\text{-a.s.}}%
		\fi
	}
\newcommand{\ase}[1]{=_{#1\text{-a.s.}}}					

\DeclareMathOperator{\cop}{copy}
\DeclareMathOperator{\discard}{del}
\newcommand{\comp}{ 		
	\mathchoice{\,}{\,}{}{} 	
}

	\providecommand{\given}{}			
	\newcommand{\SetSymbol}[1][]{%
		\nonscript\;\,#1\vert
		\allowbreak
		\nonscript\;\,
		\mathopen{}
	}
	\DeclarePairedDelimiterX{\Set}[1]{\{}{\}}{%
		\renewcommand{\given}{\SetSymbol[\delimsize]}
		#1
	}
	\makeatletter
		\let\oldSet\Set
		\def\Set{\@ifstar{\oldSet}{\oldSet*}}
	\makeatother

	\DeclarePairedDelimiterX{\Family}[1]{(}{)}{%
		\renewcommand{\given}{\SetSymbol[\delimsize]}
		#1
	}
	\makeatletter
		\let\oldFamily\Family
		\def\Family{\@ifstar{\oldFamily}{\oldFamily*}}

	\newsavebox{\numbox}%
	\newsavebox{\slashbox}%
	\newsavebox{\denbox}%
	\newlength{\slashlength}%
	\newlength{\faktorscale}%
	\DeclareDocumentCommand{\newfaktor}{m O{0.35} m O{-0.35}}{
		\savebox{\numbox}{\ensuremath{#1}}
		\savebox{\slashbox}{\ensuremath{\diagup}}
		\savebox{\denbox}{\ensuremath{#3}}
		\setlength{\faktorscale}{0.5\ht\numbox+0.5\ht\denbox}%
		\setlength{\slashlength}{2pt+0.8\faktorscale+#2\faktorscale-#4\faktorscale}%
		\raisebox{#2\ht\slashbox}{\usebox{\numbox}}
		\mkern-2mu%
		\rotatebox{-30}{\rule[#4\ht\denbox]{0.4pt}{\slashlength}}
		\mkern9mu%
		\hspace{-0.44\slashlength}%
		\raisebox{#4\ht\denbox}{\usebox{\denbox}}
	}
	\DeclareDocumentCommand{\linefaktor}{m O{0.08} m O{-0.08}}{
		\savebox{\numbox}{\ensuremath{#1}}
		\savebox{\slashbox}{\ensuremath{\diagup}}
		\savebox{\denbox}{\ensuremath{#3}}
		\setlength{\faktorscale}{0.5\ht\numbox+0.5\ht\denbox}%
		\setlength{\slashlength}{0.2\faktorscale+0.8\baselineskip}%
		\raisebox{#2\ht\slashbox}{\usebox{\numbox}}
		\mkern-1mu%
		\raisebox{-0.8pt}{%
			\rotatebox{-30}{\rule[#4\ht\denbox]{0.4pt}{\slashlength}} 
		}%
		\mkern-1mu%
		\hspace{-0.25\slashlength}%
		\raisebox{#4\ht\denbox}{\usebox{\denbox}}
	}
	
\title{The Aldous--Hoover Theorem in Categorical Probability}

\author[1]{Leihao Chen}
\author[2]{Tobias Fritz}
\author[2]{Tom{\'a}{\v{s}} Gonda}
\author[3]{Andreas Klingler}
\author[ ]{Antonio Lorenzin}

\affil[1]{Korteweg-de Vries Institute for Mathematics, University of Amsterdam, Netherlands}
\affil[2]{Department of Mathematics, University of Innsbruck, Austria}
\affil[3]{Faculty of Mathematics, University of Vienna, Austria}

\usepackage{fancyhdr}
\makeatletter
\newcommand{\runtitle}{The Aldous--Hoover Theorem in Categorical Probability}	
\makeatother
\begin{document}

\newgeometry{top=2cm,bottom=2cm,left=1in,right=1in} 
\maketitle

\begin{abstract}
	The Aldous-Hoover Theorem concerns an infinite matrix of random variables whose distribution is 
		invariant under finite permutations of rows and columns. 
	It states that, up to equality in distribution, each random variable in the matrix can be expressed as a function only depending on four key variables: 
	one common to the entire matrix, one that encodes information about its row, one that encodes information about its column, and a fourth one specific to the matrix entry.

	We state and prove the theorem within a category-theoretic approach to probability, namely the theory of Markov categories. 
	This makes the proof more transparent and intuitive when compared to measure-theoretic ones.
	A key role is played by a newly identified categorical property, the Cauchy--Schwarz axiom, which also facilitates a new synthetic de Finetti Theorem.
	
	We further provide a variant of our proof using the ordered Markov property and the d-separation criterion, both generalized from Bayesian networks to Markov categories.
	We expect that this approach will facilitate a systematic development of more complex results in the future, such as categorical approaches to hierarchical exchangeability.
\end{abstract}


\tableofcontents
\restoregeometry 

\section{Introduction}\label{sec:intro}
	
	The celebrated de Finetti Theorem is concerned with \newterm{exchangeable} probability measures $p$ on a countably infinite power $X^\N$ of a sample space $X$, where exchangeability means that $p$ is invariant under finite permutations of the factors.
	The theorem states that every such $p$ can be written as a mixture of iid measures.
	More precisely, for every exchangeable $p$ there exist a measure $q$ on some other space $A$ and a Markov kernel $f \colon A \to X$ such that
	\begin{equation}\label{eq:definetti_intro}
		\input{de_finetti_standard.tikz}
	\end{equation}
	This equation is written in the string-diagrammatic notation of categorical probability.\footnotemark
	\footnotetext{The objects $X^1, X^2, \dots$ on the right-hand side are all equal to $X$, but we use labels in superscripts to indicate how these match up with the factors of the object $X^\N = \bigotimes_{n \in \N} X$ on the left-hand side.
	This notation will become more important in the two-dimensional case of the Aldous--Hoover Theorem discussed below.
	However, we omit the superscripts in other cases when the matching of outputs is unambiguous.}
	It can be read as follows: 
	Sampling a random sequence $(x_n) \in X^\N$ from $p$ can be achieved by first sampling $a \in A$ from $q$ and then, for each $n \in \N$, sampling $x_n \in X$ from the probability measure $f( \ph | a)$.
	This recovers the traditional formulation of the de Finetti theorem in terms of conditional iid sequences of random variables \cite[Theorem 1.1]{kallenberg2005symmetries}. 

	This result is among the most foundational theorems in probability theory with substantial technical and philosophical significance.
	For example, it plays central roles in Bayesian non-parametric statistics and in the longstanding debate on the subjective vs.\ objective view of probability \cite{oneill2009exchangeability}.

	In an earlier work \cite{fritz2021definetti}, the above formulation of the de Finetti Theorem is stated and proven in a purely \emph{synthetic} fashion.
	This means that neither the statement nor the proof uses measure theory.
	Instead, both live entirely in the categorical framework of probability theory based on \newterm{Markov categories} \cite{chojacobs2019strings,fritz2019synthetic}, which axiomatizes how probabilities \emph{behave} instead of defining what probabilities \emph{are}.
	This formalism also can be instantiated on other theories of uncertainty \cite{fritz2024causal}.
	Moreover, its string-diagrammatic language provides an intuitive account of probabilistic processes and the (in)dependencies present.
	Using extra axioms with a probabilistic interpretation (e.g.\ conditionals, representability, and others) on top of the definition of Markov category gives us the structure needed to express and prove categorical versions of classical results in probability theory and statistics \cite{fritzrischel2019zeroone, fritz2023representable, fritz2021definetti, fritz2022dseparation, fritz2024hidden,moss2022ergodic,ackerman2024randomgraphs,dilavore2023evidential,pvb2024martingales}.
	For example, the synthetic version of the de Finetti Theorem recovers the traditional statement when instantiated in the Markov category $\bm{\BorelStoch}$.
	Moreover, its application to other Markov categories yields new concrete versions of the theorem \cite[Corollary 5.47]{forre2021quasimeasurable}.

	The aim of this paper is to formulate and prove a synthetic version of the \newterm{Aldous--Hoover Theorem}.
	Its classical version \cite{aldous1981representations,aldous1985exchangeability,hoover1982row} achieves a similar feat as the de Finetti Theorem, but for an infinite \emph{matrix} of random variables rather than a sequence.
	Exchangeability of the sequence is replaced by \newterm{row-column exchangeability} (also known as separate exchangeability), which means that the distribution is invariant under finite permutations of its rows and columns of the matrix.
	Using this assumption, the theorem concludes a plethora of conditional independence relations.
	These imply that the joint distribution can be decomposed in a way that is analogous to the de Finetti Theorem.
	
	Before presenting the decomposition abstractly, let us offer an intuitive explanation.
	Imagine two countably infinite groups of chess players consisting of Americans and Europeans respectively.
	They play a round-robin tournament, where each American plays against each European exactly once, so that the outcomes can be listed as in \cref{fig:rough_ah}.
	Suppose that we are interested in generating samples of what such a table might look like, provided that the only thing we know about each player is their membership in the respective team.
	So we intend to sample from a probability distribution over outcome tables which is invariant under permuting a finite number of Americans (or of Europeans), and hence row-column exchangeable.

	In such a situation, the Aldous--Hoover Theorem says the following: 
	One can randomly associate a number{\,---\,}let us call it the asymptotic Elo rating{\,---\,}to each player, such that the outcome of each game depends (probabilistically) only on the asymptotic Elo ratings of the two players together with possible external factors (e.g.\ location of the tournament, weather, etc.).
	Moreover, according to the theorem, the asymptotic Elo ratings similarly depend on the external factors only (possibly in a different way for each team).
	
	\begin{figure}[ht]
	\centering
	\begin{tikzpicture}[scale=0.55]
		\tikzsetfigurename{rough_ah_2}
		\node [person,minimum size=0.4cm] (A) at (-0.5,-1.5) {};
		\node [alice,minimum size=0.4cm] (B) at (1,0) {};
		\node [bob,minimum size=0.4cm] (C) at (2,0) {};
		\node [charlie,minimum size=0.4cm] (D) at (3,0) {};
		\node (E) at (4.5,0) {$\dots$};
		\node (eu) at (2.7,1.2) {Europeans};
		\node (ElA) at (6.5,0) {Elo};
		\node [businessman, female, minimum size=0.4cm] (A1) at (-0.5,-2.7) {};
		\node [criminal, minimum size=0.4cm](A2) at (-0.5,-3.9){};
		\node (A3) at (-0.5,-5.1) {$\vdots$};
		\node[rotate=90] (am) at (-1.7,-3.5) {Americans};
		\node[rotate=90] (ElE) at (-0.5,-7.1) {Elo};
		\draw (0.25,0.8)--(0.25,-7.8);
		\draw (5.5,0.8) -- (5.5,-7.8);
		\draw (-1.25,-0.7)--(7.3,-0.7);
		\draw (-1.25,-6.1)--(7.3,-6.1);
		\node[draw, draw opacity=0,minimum size=0.5cm] (B1) at (1,-1.5) {E};
		\node (C1) at (2,-1.5) {E};
		\node (D1) at (3,-1.5) {A};
		\node (E1) at (4.5,-1.5) {$\dots$};
		\node[draw, draw opacity=0,minimum size=0.5cm] (ElA1) at (6.5,-1.5) {1568};
		\node (B2) at (1,-2.7) {A};
		\node (C2) at (2,-2.7) {A};
		\node (D2) at (3,-2.7) {A};
		\node (E2) at (4.5,-2.7) {$\dots$};
		\node (ElA2) at (6.5,-2.7) {1203};
		\node (B3) at (1,-3.9) {E};
		\node (C3) at (2,-3.9) {E};
		\node (D3) at (3,-3.9) {D};
		\node[draw, draw opacity=0,minimum size=0.5cm] (E3) at (4.5,-3.9) {$\dots$};
		\node (ElA3) at (6.5,-3.9) {1752};
		\node[draw, draw opacity=0,minimum size=0.5cm] (B4) at (1,-5.1) {$\vdots$};
		\node (C4) at (2,-5.1) {$\vdots$};
		\node[draw, draw opacity=0,minimum size=0.5cm] (D4) at (3,-5.1) {$\vdots$};
		\node[draw, draw opacity=0,minimum size=0.7cm] (E4) at (4.5,-5.1) {$\ddots$};
		\node (ElA4) at (6.5,-5.1) {$\vdots$};
		\node[rotate=90] (ElE1) at (1,-7.1) {1375};
		\node[rotate=90] (ElE2) at (2,-7.1) {1600};
		\node[rotate=90] (ElE3) at (3,-7.1) {1804};
		\node (ElE4) at (4.5,-7.1) {$\dots$};
		\node (Ext) at (6.5,-7.1) {\includegraphics[scale=0.005]{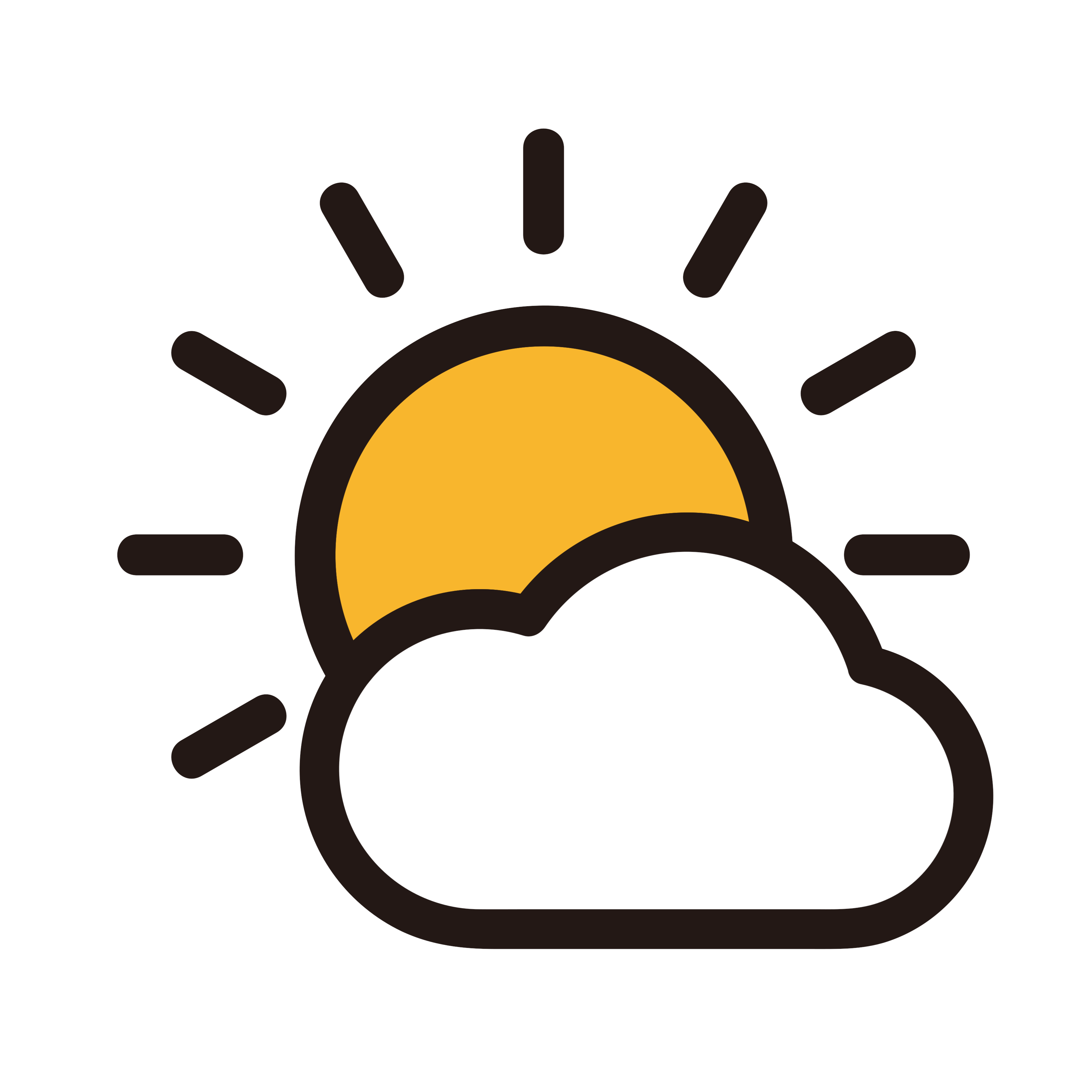}};
		\begin{scope}[on background layer]
			\draw[fill opacity=0.4, fill=blue, draw=none] (B1.north west) rectangle (B1.south east);
			\draw[fill opacity=0.4, fill=orange, draw=none] (ElA1.north west) rectangle (ElA1.south east);
			\draw[fill opacity=0.4, fill=orange, draw=none] (ElE1.north west) rectangle (ElE1.south east);
			\draw[fill opacity=0.4, fill=orange, draw=none] (Ext.north west) rectangle (Ext.south east);
		\end{scope}
	\end{tikzpicture}
	\caption{
		All the Americans play against all Europeans and we simulate the results of each game. 
		$A$ means that the American won, $D$ means a draw, and $E$ means that the European won. 
		The predicted result of a game (blue part) depends only on the orange part, i.e.\ on the simulation of the asymptotic Elo ratings of the two players involved and of the external factors.} 
	\label{fig:rough_ah}
	\end{figure}

	Turning to the formal statement, our \newterm{Synthetic Aldous--Hoover Theorem} (\cref{thm:AldousHooverWeak}) extends the de Finetti Theorem expressed via \cref{eq:definetti_intro} as follows.
	\begin{theorem*}
		In a Markov category with conditionals, countable Kolmogorov products and satisfying the Cauchy--Schwarz axiom,\footnotemark{}
		\footnotetext{See \cref{sec:prelims} for more detail on these assumptions.}%
		every row-column exchangeable morphism ${p \colon I \to X^{\mathbb{N} \times \mathbb{N}}}$ can be written as\footnotemark{}
		\footnotetext{The coloring of the wires has no significance other than being a visual aid. 
			In the main text, we avoid such complicated string diagrams and use the \newterm{plate notation} instead (\cref{sec:plate_notation}).}%
		\begin{equation}\label{eq:aldous_hoover_intro}
			\input{aldous_hoover_standard.tikz}
		\end{equation}
		for suitable morphisms $q$, $f$, $g$, and $h$.
	\end{theorem*}

	In our chess tournament example, each $X^{i,j}$ stands for the sample space $\{A,D,E\}$ of the possible results of the game of the $i$-th European against the $j$-th American.
	The morphism $q$ gives a distribution over the possible values of external factors, while $f$ and $g$ simulate the asymptotic Elo ratings of Americans and Europeans respectively as random functions of external factors.
	The simulation of the result of an individual game is performed by $h$ as a random function taking the two relevant asymptotic Elo ratings as well as the external factors as inputs.

	There is another result that is often called Aldous--Hoover Theorem.
	It concerns so-called jointly exchangeable measures, which are invariant merely when \emph{the same} permutation is applied to both rows and columns.
	The two Aldous--Hoover theorems characterize exchangeable random arrays and random graphs \cite{diaconis2008graph,lovasz2012large}, respectively.
	They have a range of applications in combinatorics \cite{austin2008exchangeable} and statistical modelling \cite{orbanz2014bayesian}.
	Moreover, they can be extended to partial exchangeability of hypergraphs and higher-dimensional arrays \cite{austin2014hierarchical,lee2022deFinetti,jung2021generalization}.
	In this paper, we focus on the Aldous--Hoover Theorem for row-column exchangeable arrays, and this is what we mean by ``Aldous--Hoover Theorem'' throughout.
	However, we believe that joint exchangeability can be also addressed in the future by similar methods.

	\subsection*{Summary}
	
	We start by providing some background material on Markov categories and conditional independence in \cref{sec:markov_cats,sec:cond_ind}, respectively.
	We then present a new \newterm{Synthetic de Finetti Theorem} in \cref{sec:definetti}.
	In comparison to the original version from \cite{fritz2021definetti}, the assumption of representability is replaced by a new information flow axiom called the \newterm{Cauchy--Schwarz Axiom} (\cref{sec:cauchy_schwarz}).
	A key step in the proof is \cref{prop:shift_perm_inv}, which establishes an equivalence between permutation-invariant and shift-invariant morphisms.
	This generalises \cite[Proposition 4.5]{fritz2021definetti} from deterministic to generic morphisms, and the Cauchy--Schwarz Axiom plays an important role here.
	Thanks to \cref{prop:shift_perm_inv} we also obtain new generalisations of the Synthetic de Finetti Theorem, namely one for dilations (\cref{thm:dF_dilation}) and one for \as{}-exchangeable morphisms (\cref{thm:dF_morphism}), with almost no extra work.
	The section then concludes by introducing the \newterm{plate notation} for string diagrams (\cref{sec:plate_notation}) and \emph{multivariate} conditional independence (\cref{sec:cond_ind_2}), both of which allow for cleaner statements.
	
	Several versions of the Synthetic Aldous--Hoover Theorem are then formulated in \cref{sec:AH_statements}, among which the strongest one is \cref{thm:AldousHoover}.
	When instantiated in measure-theoretic probability, modelled by the Markov category $\BorelStoch$, we recover the classical Aldous--Hoover Theorem for row-column exchangeable arrays of random variables taking values in standard Borel spaces (see \cref{rem:AH_traditional}).
	
	The proof of \cref{thm:AldousHoover}, illustrated in \cref{fig:strategy_ah}, begins in \cref{sec:shift_covariance}.
	The first important stepping stone is \cref{prop:shift_covariance}, which establishes that, under our assumptions, every permutation covariant morphism is also shift covariant. 
	The Cauchy--Schwarz Axiom is once again crucial to obtain this foundational result. 
	Next, our Synthetic de Finetti Theorem and the existence of conditionals are used to show a number of conditional independence relations in \cref{sec:3_cond_ind}.
	Finally, in \cref{sec:mainProof} we put these ingredients together to obtain the complete proof.
	
	In \cref{sec:markov_prop}, we present a more systematic variant of the proof by introducing the \newterm{ordered Markov property} and using the d-separation criterion \cite{fritz2022dseparation}. 
	Although less direct, we believe that this approach will facilitate a characterization of hierarchical exchangeability for higher-dimensional arrays of random variables \cite{austin2014hierarchical,lee2022deFinetti,jung2021generalization} in the context of Markov categories.

	In \cref{sec:param}, we recall parametric Markov categories and prove that the Cauchy--Schwarz axiom is stable under moving to a parametric Markov category.
	This allows for parametric versions of the de Finetti and Aldous--Hoover Theorems, which are stated in \cref{sec:dF_further,sec:AH_further} respectively. 
	These two appendices further investigate a dilational and an almost sure version for both statements.
	As far as we know, these versions of the Aldous--Hoover Theorem are new even for standard Borel spaces.

	The final appendix (\cref{sec:omnirandom}) introduces the concept of omnirandomness, which we employ in the main text to derive the functional form of the Aldous--Hoover theorem (\cref{thm:AH_functional}).
	
	\subsection*{Acknowledgements}
	We thank Dario Stein for suggesting plate notation to us.
	This research was funded in part by the Austrian Science Fund (FWF) [doi:\href{https://www.doi.org/10.55776/P35992}{10.55776/P35992}, doi:\href{https://www.doi.org/10.55776/P33122}{10.55776/P33122}, doi:\href{https://www.doi.org/10.55776/Y1261}{10.55776/Y1261}]. For open access purposes, the authors have applied a CC BY public copyright license to any author accepted manuscript version arising from this submission. 
	Leihao Chen acknowledges Booking.com for support.
	Andreas Klingler further acknowledges funding of the Austrian Academy of Sciences (\"OAW) through the DOC scholarship 26547.

\section{Preliminaries} 
\label{sec:prelims}

	This section covers the preliminaries necessary to state and prove our Synthetic Aldous--Hoover Theorem. 
	In \cref{sec:markov_cats} we briefly review the relevant theory of Markov categories. 
	An essential tool for our proofs is conditional independence, which we recall in \cref{sec:cond_ind}. 
	\cref{sec:cauchy_schwarz} is devoted to the Cauchy--Schwarz axiom, which is the only assumption present in our theorems that is original to this work. 
	In \cref{sec:definetti} we employ this axiom to prove the new Synthetic de Finetti Theorem. 
	
	As the informal statement of the Synthetic Aldous--Hoover Theorem from the introduction illustrates, it can be confusing, if not misleading, to consider traditional string diagrams with infinitely many wires.
	In order to address this problem, \cref{sec:plate_notation} introduces plate notation for Markov categories.
	A first application of this notation is given in \cref{sec:cond_ind_2}, where we extend the notion of conditional independence to infinite families of objects.
	We also prove the \nameref{lem:partition}{\,---\,}a result that is important for the three independence lemmas discussed in \cref{sec:3_cond_ind}.
		
	\subsection{Markov Categories and Relevant Assumptions}
	\label{sec:markov_cats}
	
	We start with a brief account of Markov categories before introducing new concepts.
	For a more detailed treatment, we refer the reader to previous works \cite{fritz2019synthetic,fritz2021definetti,perrone2024markov}. 
	In essence, a \newterm{Markov category} is a symmetric monoidal category $(\cC, \otimes , I)$ such that the monoidal unit $I$ is terminal, and every object $X$ comes equipped with a commutative comonoid structure in a way that is compatible with the monoidal structure.
	The comonoid structure morphisms are called \newterm{copy}, denoted $\cop_X \colon X \to  X \otimes X$, and \newterm{delete}, denoted $\discard_X \colon X \to I$.
	In string diagrams, we write them as
	\begin{equation}
		\input{copy_del.tikz}
	\end{equation}
	It is important to keep in mind that the copy morphisms are not natural. A morphism $f$ with respect to which copying is natural, meaning that $f$ commutes with copy, is called \newterm{deterministic}.
	These morphisms form a cartesian monoidal subcategory of $\cC$.
	Additionally, morphisms whose domain is the monoidal unit $I$ are referred to as \newterm{states}.
	
	The central example for this work is the Markov category $\BorelStoch$, whose objects are standard Borel spaces, morphisms are Markov kernels composed via the Chapman--Kolmogorov equation, and the monoidal structure is given by the usual products of measurable spaces (on objects) and Markov kernels (on morphisms).
	Its deterministic morphisms are exactly those Markov kernels which correspond to measurable maps, and we do not distinguish between a measurable map and its associated Markov kernel.
	With this in mind, $\cop_X \colon X \to X \otimes X$ is the diagonal map of $X$.
	States in $\BorelStoch$ are probability measures.
	Composing with the morphism $\discard_X \colon X \to I$ implements marginalization, in the sense that for any $p \colon I \to X \otimes Y$, the composite $(\discard_X \otimes \id_Y) \circ p$ is the marginal of the joint distribution $p$ on $Y$, drawn as
	\begin{equation}
		\input{marginal_p.tikz}
	\end{equation}
	
	Conditioning is a fundamental concept in probability theory. 
	In a Markov category, a conditional of a morphism $f \colon A \to X \otimes Y$ given $X$ is any morphism $f_{|X} \colon X \otimes A \to Y$ satisfying
	\begin{equation}
		\input{conditional.tikz}
	\end{equation}
	A Markov category is said to have \newterm{conditionals} if every $f$ as above has a conditional.
	In $\BorelStoch$, conditionals exist, and for $A = I$ they correspond to regular conditional probabilities.
	The existence of conditionals plays a key role in this work, as it implies the semigraphoid property known as \emph{weak union} for conditional independence relations in Markov categories (see \cref{sec:cond_ind}), which we use repeatedly in our proofs.
	
	Both the de Finetti and Aldous--Hoover Theorems concern infinitely many random variables. 
	In Markov categories, we correspondingly work with states on infinite tensor products of objects.
	These are formalized as {Kolmogorov products}, as introduced in \cite{fritzrischel2019zeroone}.
	Given a family of objects $(X_i)_{i \in \mathcal{I}}$, let us consider the collection of finite tensor products $X^F=\bigotimes_{i \in F} X_i$, where $F \subseteq \mathcal{I}$ is any finite subset, together with marginalization maps
	\[ 
		\pi_{F,G} = \id_{X^F} \otimes \discard_{X^{F\setminus G}} \colon X^F \to X^G
	\] 
	between them for any pair of finite subsets $F,G \subseteq \mathcal{I}$ satisfying $F \supseteq G$.
	The categorical limit of this diagram, denoted by $X^{\mathcal{I}}$, is called a \newterm{Kolmogorov product} if it is preserved by the tensoring functor $\ph \otimes Y$ for every object $Y$, and if all morphisms $X^{\mathcal{I}} \to X^{F}$ in the limit cone are deterministic.
	This terminology is motivated by the close connection with the Kolmogorov extension theorem, which states that joint probability measures of infinitely many random variables are in bijection with compatible families of joint probability measures for each finite subset of variables.
	By considering objects in place of random variables and states in place of probability measures, this result translates into a bijection between the set of states $I \to X^{\mathcal{I}}$ and the set of compatible families of states $(I \to X^F)_{F\subseteq \mathcal{I}\text{ finite}}$.
	In fact, the existence of Kolmogorov products is ensured by a parametric version of the Kolmogorov extension theorem; see \cite{fritzrischel2019zeroone} for details.	
	
	These axioms were previously used in the synthetic proof of the de Finetti Theorem \cite{fritz2021definetti}, alongside an additional assumption of representability \cite{fritz2023representable}. 
	Instead of the latter, we now make use of the Cauchy--Schwarz axiom, which we introduce and discuss in \cref{sec:cauchy_schwarz}.
	
	\begin{assumption}\label{ass:de_finetti}
		Throughout the paper, we assume that $\cC$ is a Markov category that
		\begin{enumerate}
			\item has conditionals,
			\item has countable Kolmogorov products, and
			\item satisfies the Cauchy--Schwarz axiom (\cref{def:cauchy_schwarz}).
		\end{enumerate}
	\end{assumption}
	These three requirements are the only necessary ingredients in the proof of our synthetic Aldous--Hoover Theorem.
	For the sake of clarity, we restate them explicitly in our main result.

	By employing an axiomatic approach, this paper replaces long measure-theoretic calculations by the more visual tool of string diagrams.
	However, measure-theoretic aspects are certainly not circumvented, since they are required to prove that these axioms hold for the Markov categories capturing measure-theoretic probability.
	For instance, proving that these Markov categories have conditionals requires a parametric disintegration of measures, a deep result that relies on powerful measure-theoretic machinery.
	We provide references to the proofs in the literature in the following example; spelling them out here would require a lengthy excursion into measure theory.
	
	\pagebreak[3]
	\begin{example}[Markov categories satisfying the assumptions]\label{ex:assumptions}\hfill
		\begin{enumerate}
			\item The main example we have in mind is $\BorelStoch$, the Markov category of standard Borel spaces and Markov kernels. 
			It satisfies \cref{ass:de_finetti}, which can be shown by combining \cite[Example~11.7]{fritz2019synthetic}, \cite[Example 3.6]{fritzrischel2019zeroone}, and \cref{lem:stoch_cs}.
		\item  Let $\AnaStoch$ and $\UnivStoch$ be the full Markov subcategories of $\Stoch$ where we restrict to analytic measurable spaces\footnote{Also known as Souslin spaces.} \cite[Section 6.6]{bogachev2007measure} and universally measurable spaces \cite[Definitions B.1, B.2 and B.27]{forre2021conditional}\footnote{For an earlier use of the term, see \cite[p.~155]{kechris}.}, respectively.
			By definition, we have full Markov subcategory inclusions
			\[
				\BorelStoch \subseteq \AnaStoch \subseteq \UnivStoch \subseteq \Stoch.
			\]
			All four categories satisfy the Cauchy-Schwarz axiom by \cref{lem:stoch_cs}.
			To prove that the first three have countable Kolmogorov products, one can use the Kolmogorov Extension Theorem for universally measurable spaces \cite[Theorem 12.1.2]{dudley2018real}, the argument for the universal property from \cite[\mbox{Example 3.6}]{fritzrischel2019zeroone}, and the fact that both analytic and universally measurable spaces are closed under countable products \cite[Lemma B.29]{forre2021conditional}.
			They also have conditionals, as proven in \cite[Theorem 3.4]{bogachev20kant} and \cite[Theorems C.6 and C.7]{forre2021conditional}. 
			Thus, $\AnaStoch$ and $\UnivStoch$ also satisfy \cref{ass:de_finetti}.

		\end{enumerate}
	\end{example}
	
	\subsection{Conditional Independence}
	\label{sec:cond_ind}
		
		A key aspect of the significance of the de Finetti and Aldous--Hoover Theorems lies in the specific structure of conditional independences they establish.
		This is also relevant for the proof, which can be conducted by assembling several conditional independence relations to derive the final result \cite{aldous1985exchangeability}.
		Since we adopt a similar strategy in our approach (see \cref{fig:strategy_ah}), we give some background on conditional independence in Markov categories, as been introduced and studied in \cite[Section 6.2]{chojacobs2019strings} and \cite[Section~12]{fritz2019synthetic}.

		\begin{definition}
			\label{def:cond_ind}
			A morphism $p\colon I \to X \tensor W \tensor Y$ in a Markov category $\cC$ displays the \newterm{conditional independence}
			\begin{equation}
				X\perp Y \mid W
			\end{equation}
			if there exist morphisms $f\colon W \to X$ and $g \colon W \to Y$ satisfying
			\begin{equation}\label{eq:cond_ind}
				\input{display_cond_ind.tikz}
			\end{equation}
		\end{definition}

		Such $f$ is necessarily a conditional of the marginal of $p$ on $X \tensor W$ given $W$, and similarly for $g$ \cite[Remark 12.2]{fritz2019synthetic}.
		
		Whenever a marginal of $p$ displays a specific conditional independence, we also say that $p$ itself displays such a conditional independence. 
		This causes no confusion as long as the outputs have distinct names, since the relevant outputs are always explicitly specified.
		For instance, saying that the morphism $p$ from \cref{def:cond_ind} displays $X \perp Y \mid I$ (which is just plain independence) means that it satisfies
		\begin{equation}\label{eq:display_cond_ind_2}
			\input{display_cond_ind_2.tikz}
		\end{equation}
		for some states $f$ and $g$.
		
		\begin{notation}
			\label{nota:conditional_independence}
			For easier reading, we generally denote a tensor product within a conditional independence relation by a comma, e.g.\ writing $X_1 , X_2 \perp Y \mid W$ instead of $X_1 \otimes X_2 \perp Y \mid W$.
		\end{notation}
		
		One of the advantages of working with conditional independence relations is that they are more concise than fully writing out statements like \cref{eq:cond_ind}.
		Deriving implications between such relations is key in our proofs.
		Important ingredients for them are the following well-known properties of conditional independence, which have been proven for Markov categories in \cite[Lemma 12.5]{fritz2019synthetic}.

		\begin{lemma}[Semigraphoid properties]\label{lem:semigraphoid}
			Let $\cC$ be a Markov category with conditionals.
			Then the following implications between conditional independence relations hold for any state in $\cC$:
			\begin{enumerate}[label=(\roman*)]
				\item\label{semi:symmetry} \newterm{Symmetry:}
				\begin{equation}
					X \perp Y \mid W \quad \Rightarrow \quad Y \perp X \mid W;
					\end{equation}
				\item\label{semi:decomposition} \newterm{Decomposition:}
				\begin{equation}
					X_1, X_2 \perp Y \mid W\quad \Rightarrow\quad X_1 \perp Y\mid W;
					\end{equation}
				\item\label{semi:contraction} \newterm{Contraction:}
				\begin{equation}
					X\perp Y \mid Z , W\ \ \text{and}\ \ X \perp Z \mid W \quad \Rightarrow \quad X \perp Z , Y \mid W;
					\end{equation}
				\item\label{semi:weak_union} \newterm{Weak union:}
				\begin{equation}
					X_1 , X_2 \perp Y \mid W \quad \Rightarrow \quad X_1 \perp Y \mid W , X_2.
				\end{equation}
			\end{enumerate}
		\end{lemma}
		
		We use the shorthands ``Dec'' for decomposition and ``WU'' for weak union to indicate which semigraphoid property we employ in our proofs.
		
		In \cref{sec:cond_ind_2}, we introduce an extended definition of conditional independence between more than three outputs, which we use throughout our work.

	\subsection{The Cauchy--Schwarz Axiom}
	\label{sec:cauchy_schwarz}

	
	\begin{definition}\label{def:cauchy_schwarz}
		A Markov category $\cC$ satisfies the \newterm{Cauchy--Schwarz axiom} if for all $A,X,Y \in \cC$, the implication
		\begin{equation}\label{eq:spat_causal}
			\input{spat_causal.tikz}
		\end{equation}
		holds for all morphisms $p\colon A \to X$ and $f,g \colon X \to Y$.
	\end{definition}
	The consequent of this implication says that $f$ and $g$ are $\as{p}$ equal \cite[Definition 13.1]{fritz2019synthetic}.

	Recall that $\Stoch$ is the Markov category with measurable spaces as objects and measurable Markov kernels as morphisms \cite[Section 4]{fritz2019synthetic}.
	Since $\BorelStoch$, $\AnaStoch$ and $\UnivStoch$ (see \cref{ex:assumptions}), are full Markov subcategories of $\Stoch$, the following lemma also shows that these categories satisfy the Cauchy--Schwarz axiom.
	\begin{lemma}
		\label{lem:stoch_cs}
		$\Stoch$ satisfies the Cauchy--Schwarz axiom.
	\end{lemma}

	\begin{proof}
		For any $a\in A$ and for any measurable subset ${D \subseteq Y}$, the antecedent of Implication \eqref{eq:spat_causal} reads
		\begin{equation}
			\label{eq:stoch_cs_ass}
			\int_X f(D  |  x)^2 \, p(\dr x |  a) = \int_X f(D |  x) \, g(D |  x) \, p(\dr x |  a) = \int_X g(D |  x)^2 \, p(\dr x |  a),
		\end{equation}
		which implies
		\begin{equation}
		    \int_X \bigl( f(D |  x) - g(D |  x) \bigr)^2 \, p(\dr x |  a) = 0.
		\end{equation}
		Since the integrand is nonnegative, it must vanish almost everywhere, resulting in
		\begin{equation}
		    \int_E f(D |  x) \, p(\dr x |  a) = \int_E g(D |  x) \, p(\dr x |  a).
		\end{equation}
		for every measurable $E \subseteq X$.
		This proves the desired statement by the existing characterization of almost sure equality in $\Stoch$~\cite[Example~13.3]{fritz2019synthetic}.
	\end{proof}

	\begin{remark}
		The name ``Cauchy--Schwarz axiom'' originates from a slightly different line of argument: omitting the $D$ and using the inner product in $L^2(X, p(\ph|a))$, the \cref{eq:stoch_cs_ass} can be written as 
		\begin{equation}
			\label{eq:stoch_cs_ass2}
			\langle f,f \rangle = \langle f,g \rangle = \langle g,g \rangle.
		\end{equation}
		This implies $\lvert \langle f,g\rangle \rvert^2 = \langle f,f\rangle \langle g,g \rangle$ in particular, and therefore $f$ and $g$ differ only up to a scalar multiple in $L^2(X, p(\ph|a))$ by the equality criterion for the Cauchy--Schwarz inequality.
		But then another appeal to \cref{eq:stoch_cs_ass2} shows that this scalar is $1$, resulting in the desired $f \ase{p} g$.

		More informally, we also think of morphisms from the antecedent of \cref{eq:spat_causal} as categorical analogues of the inner product expressions from \cref{eq:stoch_cs_ass2}. 
	\end{remark}

\subsection{A New Synthetic de Finetti Theorem}
\label{sec:definetti}

	We now prove a variant of the Synthetic de Finetti Theorem from \cite{fritz2021definetti}.
	The key difference is that we replace the assumption of \as{}-compatible representability by the Cauchy--Schwarz axiom.
	In other words, we work under \cref{ass:de_finetti}.
	This allows us to use this result in our proof of the Synthetic Aldous--Hoover Theorem, in which we do not assume representability either.\footnote{It is worth noting that we have not been able to prove the Aldous--Hoover Theorem without the Cauchy--Schwarz axiom, even if \as{}-compatible representability is assumed instead.}

	In order to define the morphisms of interest for the de Finetti Theorem, we recall the notation for permuting (and discarding) selected elements in a Kolmogorov power $X^{\N}$~\cite[Section 4]{fritz2021definetti}.
	Namely given an injective map $\varphi \colon \N \to \N$, we define ${X^{\varphi} \colon X^{\N} \to X^{\N}}$ in terms of the universal property of the Kolmogorov product as the unique endomorphism which maps the $n$-th factor to the $\varphi^{-1}(n)$-th factor whenever $n$ is in the image of $\varphi$ and discards it otherwise.\footnotemark{} 
	\footnotetext{As for exponential objects in cartesian closed categories, the mapping $\varphi \mapsto X^{\varphi}$ is contravariant in $\varphi$.}%
	This notation is primarily used for a permutation of wires $X^{\sigma}$, where $\sigma \colon \N \to \N$ is a permutation, as well as for the successor function $s \colon \N \to \N$ mapping $n$ to $n+1$.
	For example, for the permutation $\sigma = (1 \ 2 \ 3)$ and successor function $s$, we have\footnotemark{}
	\footnotetext{As in \cite{fritz2021definetti}, a double wire indicates that the respective object is a Kolmogorov product (finite or infinite).
		This becomes especially helpful when parsing diagrams with plates (see \cref{sec:plate_notation}).}%
	\begin{equation}\label{eq:perm_shift}
		\input{perm_shift.tikz}
	\end{equation}
	
	\begin{definition}
		A morphism $f \colon A\to X^{\N}$ is \newterm{exchangeable} if, for all finite permutations $\sigma \colon \N \to \N$, we have
		\begin{equation}\label{eq:exchangeable_def}
			\input{exchangeable_def.tikz}
		\end{equation}
	\end{definition}

	The following preliminary result resembles \cite[Proposition 4.5]{fritz2021definetti}, but unlike that proposition, it does not require $f$ to be $\as{p}$ deterministic. 
	We remind the reader that we are working under \cref{ass:de_finetti}, even though the existence of conditionals is not necessary for the following result.

	\begin{proposition}[Permutation invariance is equivalent to shift invariance]\label{prop:shift_perm_inv}
		Let ${p\colon A \to X^{\N}}$ be an exchangeable morphism and let ${f \colon X^{\N} \to Y}$ be any morphism. 
		Then the following are equivalent:
		\begin{enumerate}
			\item\label{it:perm_inv} Permutation invariance: For every finite permutation $\sigma$, we have
			\begin{equation} 
				f \ase{p} f \comp X^{\sigma}.
			\end{equation}

			\item\label{it:shift_inv} Shift invariance: For the successor function $s \colon \N \to \N$ given by $n \mapsto n+1$, we have
			\begin{equation}
				f \ase{p} f \comp X^{s}.
			\end{equation}
		\end{enumerate}
	\end{proposition}
	\begin{proof}
		The proof of \ref{it:shift_inv} $\Rightarrow$ \ref{it:perm_inv} is the same as in the proof of \cite[Proposition 4.5]{fritz2021definetti}: 
		For a given finite permutation $\sigma$, the claim follows upon choosing $n \in \N$ satisfying $\sigma \comp s^n = s^n$, for which we then have $(X^s)^n \comp X^\sigma = (X^s)^n$.

		We still need to prove \ref{it:perm_inv} $\Rightarrow$ \ref{it:shift_inv}. 
		We thus assume that $f$ is invariant under finite permutations and that $p$ is exchangeable.
		Therefore, we can use the fact that $X^\sigma$ is deterministic for every finite permutation $\sigma$ to obtain
		\begin{equation}\label{eq:shift_perm_inv_proof1}
			\input{shift_perm_inv_proof1.tikz}
		\end{equation}
		so that the morphism on the left is exchangeable in its second output. 
		Therefore, by \cite[Lemma 5.1]{fritz2021definetti}, it is also invariant under applying the shift $X^s$ instead of a permutation.\footnotemark{}
		\footnotetext{This result and its proof hold in any Markov category with countable Kolmogorov products.}%
		Using this fact together with the shift invariance of $p$ and the fact that $X^s$ is deterministic gives us
		\begin{equation}\label{eq:shift_perm_inv_proof2}
			\input{shift_perm_inv_proof2.tikz}
		\end{equation}
		We then obtain the desired shift invariance $f \ase{p} f \comp X^s$ by applying the Cauchy--Schwarz axiom.
	\end{proof}

	Given an exchangeable state $p\colon I \to X^{\N}$, we can consider the conditional of the first output given the remaining ones, i.e.\ a morphism $p_{|\tail}\colon X^{\N} \to X$ satisfying the first equation (the second one uses the shift invariance of $p$) in
	\begin{equation}\label{eq:tail_conditional_def}
		\input{tail_conditional_def.tikz}
	\end{equation}
	which always exists in a Markov category with conditionals.
	We call a choice of such a morphism the \newterm{tail conditional} of $p$. 

	\begin{theorem}[Synthetic de Finetti Theorem, strong form]
		\label{thm:definetti}
		Let $\cC$ be a Markov category satisfying \cref{ass:de_finetti}.
		Then every exchangeable state $p\colon I \to X^{\N}$ can be written in the form
		\begin{equation}\label{eq:exc_cond_iid}
			\input{exc_cond_iid.tikz}
		\end{equation}
		for any $n \in \N$, where $\compl{\A} \coloneqq \mathbb{N}\setminus \lbrace 1,2,\dots, n \rbrace$ refers to the tail of the infinite sequence of outputs.
	\end{theorem}

	\begin{proof}
		We follow the argument given in \cite{fritz2021definetti} and merely indicate what needs to be changed.
		The proof of \cite[Lemma 5.3]{fritz2021definetti} simplifies to just noting that the tail conditional is invariant under finite permutations and applying \cref{prop:shift_perm_inv}.
		Therefore, \cite[Equation (43)]{fritz2021definetti} also holds, which is exactly the present claim.
	\end{proof}

	The following version of the de Finetti Theorem mirrors the measure-theoretic version \cite[Theorem 1.1]{kallenberg2005symmetries}. 
	\begin{corollary}[Synthetic de Finetti Theorem, weak form]
		\label{cor:definetti}
		Let $\cC$ be a Markov category satisfying \cref{ass:de_finetti}.
		Then for every exchangeable state $p\colon I \to X^{\N}$, there exist an object $A \in \cC$ together with morphisms $q \colon I \to A$ and $f \colon A \to X$ such that
		\begin{equation}\label{eq:de_finetti_weak}
			\input{de_finetti_weak.tikz}
		\end{equation}
	\end{corollary}

	\begin{proof}
		This follows from \cref{thm:definetti} by the universal property of the Kolmogorov power $X^\N$, since with $A \coloneqq X^{\N}$, $q \coloneqq p$, and $f \coloneqq p_{|\tail}$, \cref{eq:exc_cond_iid} marginalizes to the $X^{\mathcal{A}}$-marginal of \cref{eq:de_finetti_weak} for any $\mathcal{A} \coloneqq \{ 1, 2, \dots, n \}$. 
	\end{proof}

\subsection{Plate Notation}
\label{sec:plate_notation}

	In this section, we adapt the \newterm{plate notation} for Bayesian networks~\cite{buntine1994operations} to string diagrams in Markov categories.
	This serves as a shorthand notation for conditionally independent application of multiple morphisms. 	
	Given a family of morphisms $\left( f_i \colon A \to X_i \right)_{i \in \mathcal{I}}$ indexed by a set $\mathcal{I}$ for which the Kolmogorov product $X^{\mathcal{I}} \coloneqq \bigotimes_{i \in \mathcal{I}} X_i$ exists, plate notation is
	\begin{equation}\label{eq:plate_definition}
		\input{plate_definition.tikz}
	\end{equation}
	where we leave it understood that the right-hand side denotes the unique morphism whose marginal on any finite subset $\{i_1, \dots, i_n\} \subseteq \mathcal{I}$ is as indicated.\footnote{We need not assume that $\mathcal{I}$ is countable, as the diagram on the right-hand side of \cref{eq:plate_definition} may suggest.}
	Alternatively,
	we may write this as
	\begin{equation}
		\left( \vphantom{\bigotimes} \smash{\bigotimes_{i \in \mathcal{I}} f_i} \right) \comp \cop_A,
	\end{equation}
	where $\cop_A$ denotes the canonical $\mathcal{I}$-fold copy $A \to A^{\mathcal{I}}$; but such definition also requires the existence of the Kolmogorov power $A^{\mathcal{I}}$ in addition to $X^{\mathcal{I}}$.

%
	
	Using plate notation, the statement of our Synthetic de Finetti Theorem via \cref{eq:exc_cond_iid} reads as follows: 
	For $\A = \{1,\dots,n\}$, we have
	\begin{equation}\label{eq:deFinetti_plate1}
		\input{deFinetti_plate1.tikz}
	\end{equation}
	The morphism $p_{|\tail}$ inside the plate carries no index because it does not vary with the index $i$.
	
	
	
	Since plate diagrams depict morphisms in the Markov category, we can compose them in sequence and in parallel.
	For instance, given two index sets $\mathcal{I} = \mathcal{J} = \{1,2\}$, we have the following graphical representation of sequential composition:
	\begin{equation}\label{eq:plate_composition}
		\input{plate_composition.tikz}
	\end{equation}
	and parallel composition:
	\begin{equation}\label{eq:plate_composition_2}
		\input{plate_composition_2.tikz}
	\end{equation}
	
	For the Aldous--Hoover Theorem, we use overlapping plates in order to represent matrices of morphisms with two indices.
	This notation can be defined using sequential and parallel composition of plates.

	\begin{notation}[Overlapping plates]\label{not:plate_overlap}
		Given families of morphisms
		\begin{gather*}
			(f_i \colon A \to X_i \otimes V_i)_{i \in \mathcal{I}}, \qquad (g_j \colon B \to W_j \otimes Y_j)_{j \in \mathcal{J}},	\\[2pt]
			(h_{ij} \colon V_i \otimes W_j \to Z_{ij})_{i \in \mathcal{I}, \, j \in \mathcal{J}}
		\end{gather*}
		we introduce the following diagrammatic notation:\footnote{We implicitly use the natural isomorphisms $(X \otimes V)^{\mathcal{I}} \cong X^{\mathcal{I}} \otimes V^{\mathcal{I}}$ and $(W \otimes Y)^{\mathcal{J}} \cong W^{\mathcal{J}} \otimes Y^{\mathcal{J}}$.}
		\begin{equation}\label{eq:plate_overlap}
			\input{plate_overlap.tikz}
		\end{equation}
		where the morphisms $\pi_i \colon V^{\mathcal{I}} \to V_i$ and $\pi_j \colon W^{\mathcal{J}} \to W_j$ are the projections of the Kolmogorov products, discarding all inputs except the $i$-th and $j$-th instances, respectively.
	\end{notation}
	The triple wires in \cref{eq:plate_overlap} are used to distinguish Kolmogorov products of doubly indexed families from those with a single index, which are typically represented by a double wire. 
	Our intention is to remind the reader that when triple wires are used, the resulting outputs should be interpreted as entries of a matrix, whether finite or infinite in size.
	
	For an explicit example of the overlapping plate notation, see \cref{fig:AH_2x2}.

%
%


\subsection{Multivariate Conditional Independence}
\label{sec:cond_ind_2}
	
	The conditional independence appearing in the de Finetti Theorem (\cref{thm:definetti}) and its various extensions cannot be stated as a single instance of \cref{def:cond_ind}.
	We can write them either as a family of basic conditional independence relations (see \cref{cor:partition}) or via the following more general notion.
	
	\begin{definition}
		\label{def:mut_cond_ind}
		Consider a family of objects $(X_i)_{i \in \A}$ such that the Kolmogorov product $X^\A$ exists. 
		We say that a state $p\colon I \to X^\A \tensor W$ displays the \newterm{conditional independence} $\perp_{i \in \A} X_i \mid W$ if there is a family of morphisms $(f_i \colon W \to X_i)_{i \in \A}$ satisfying
		\begin{equation}
			\label{eq:display_cond_ind_n}
			\input{display_cond_ind_n.tikz}
		\end{equation}
	\end{definition}

	\begin{remark}
		\label{rem:mut_cond_ind}
		As a straightforward marginalization of the equation shows, each $f_i$ in \cref{eq:display_cond_ind_n} is necessarily a conditional of the marginal of $p$ on $X_i \otimes W$ given $W$ \cite[Remark~12.2]{fritz2019synthetic}.
	\end{remark}

	Let us give an example of such a multivariate conditional independence relation.
	\cref{thm:definetti} shows that every exchangeable state $p\colon I \to X^{\N}$ displays the conditional independence\footnotemark{}
	\footnotetext{In \cref{eq:definetti_cond_ind}, superscripts are used instead of subscripts since the outputs in the context of de Finetti Theorem are necessarily isomorphic and the index $i$ merely communicates their position in the infinite sequence (see \cref{nota:ah}).}%
	\begin{equation}\label{eq:definetti_cond_ind}
		\bigperp_{i = 1, \dots, n} X^i \largemid X^{\compl{\A}}
	\end{equation}
	for every $n \in \N$.
	In fact, this conditional independence can be thought of as the main content of the theorem, since the theorem can be easily derived from it by combining \cref{rem:mut_cond_ind} with the shift invariance of $p$.

	Let $X^\A$ be a Kolmogorov product, and consider a partition $P$ of the set $\A$, whose elements are thus disjoint subsets of $\A$. 
	Then we can write $X^{\A}$ as a Kolmogorov product over the set $P$ itself via the obvious natural isomorphism
	\begin{equation}
		X^\A \cong \bigotimes_{\I \in P} X^\I.
	\end{equation}
	In particular, we can ask whether a state $I \to X^{\A} \otimes W$ displays conditional independence ${\perp_{\I \in P} X^\I \mid W}$. 

	\begin{lemma}[Partition Lemma]\label{lem:partition}
		Consider a morphism $p\colon I \to X^{\A} \otimes W$ and two partitions $P_1$ and $P_2$ of $\A$ such that the conditional independence relations
		\begin{equation}\label{eq:partition_assumption}
			\bigperp_{\I \in P_1} X^\I \mid W, \qquad \bigperp_{\J \in P_2} X^\J \mid W
		\end{equation}
		hold.
		Then, also the conditional independence
		\begin{equation}\label{eq:partition_statement}
			\bigperp_{\I \in P_1 \vee P_2} X^\I \mid W
		\end{equation}
		holds, where $P_1 \vee P_2$ is the partition obtained by intersecting each element of $P_1$ with each element of $P_2$.
	\end{lemma}
	\begin{proof}
		Consider a subset $\I \subseteq \A$ that appears in the partition $P_1$.
		Marginalizing the $X_i$ for $i \in \A \setminus \I$ and using the conditional independence relations from \eqref{eq:partition_assumption}, we obtain
		\begin{equation}\label{eq:partition_cond_ind_proof1}
			\input{partition_cond_ind_proof1.tikz}
		\end{equation}
		for a suitable family $(g_\J \colon W \to X^{\J})_{\J \in P_2}$, where the output of each $g_\J$ is factored into a tensor product of $X^{\I \cap \J}$ and $X^{\J \setminus (\I \cap \J)}$, and the latter is discarded.
		We can apply this for each $\I \in P_1$ to get
		\begin{equation}\label{eq:partition_cond_ind_proof2}
			\input{partition_cond_ind_proof2.tikz}
		\end{equation}
		which concludes the proof.
	\end{proof}

	\begin{corollary}\label{cor:partition}
		Consider a morphism $p\colon I \to X^{\A}\otimes W$ for a finite set $\A$.
		Then we have 
		\begin{equation}
			\bigperp_{i \in \A} X_i \mid W \qquad \iff \qquad  X_i \perp X^{\A \setminus\lbrace i \rbrace} \mid W \quad \forall i\in \A.
		\end{equation}
	\end{corollary}
	\begin{proof}
		The ($\Rightarrow$) direction holds by definition, while the converse is an application of \cref{lem:partition}.
	\end{proof}

\section{A Synthetic Aldous--Hoover Theorem}
\label{sec:aldous_hoover}

	\subsection{Statements of the Theorem}
	\label{sec:AH_statements}

	We begin the main developments of this paper by defining the class of morphisms to be characterized by the Synthetic Aldous--Hoover Theorem.

	\begin{definition}
		A morphism $p \colon I \to X^{\N \times \N}$ is \newterm{row-column exchangeable} if for every finite permutation $\sigma \colon \N \to \N$, we have 
		\begin{equation}\label{eq:rce_condition}
			\input{rce_condition.tikz}
		\end{equation}
	\end{definition}

	The use of triple wires in \cref{eq:rce_condition} expresses that the outputs are arranged in a matrix, as discussed at the end of \cref{sec:plate_notation}.
	The remainder of this section is devoted to a proof of a stronger version (\cref{thm:AldousHoover}) of the following key result.

	\begin{theorem}[Synthetic Aldous--Hoover Theorem, weak form]\label{thm:AldousHooverWeak}
		Let $\cC$ be a Markov category satisfying \cref{ass:de_finetti}.
		Then for every row-column exchangeable state $p \colon I \to X^{\N \times \N}$, there exist objects $A, B, C \in \cC$ and morphisms
		\begin{equation}
			q \colon I \to A, \qquad f \colon A \to B, \qquad g \colon A \to C, \qquad h \colon B \otimes A \otimes C \longrightarrow X
		\end{equation}
		such that we have
		\begin{equation}
			\input{aldous_hoover_standardplate.tikz}
		\end{equation}
	\end{theorem}

	Before we turn to the proof, let us indicate how this result can also be phrased in a functional form based on the concept of omnirandom states (\cref{def:omnirandom}).

	\begin{corollary}[Synthetic Aldous--Hoover Theorem, functional form]\label{thm:AH_functional}
		Let $\cC$ be a Markov category satisfying \cref{ass:de_finetti} and having an omnirandom state $r \colon I \to R$.
		Then for every row-column exchangeable morphism $p \colon I \to X^{\N \times \N}$, there exists a deterministic $k \colon R^{\otimes 4} \to X$ such that we have
		\begin{equation}
			\label{eq:aldous_hoover_functional}
			\input{aldous_hoover_functional.tikz}
		\end{equation}
	\end{corollary}

	\begin{proof}
		By applying omnirandomness to the morphisms $q$, $f$, $g$, and $h$ as in the statement of \cref{thm:AldousHooverWeak}, we obtain
		\begin{equation}
			\input{aldous_hoover_functional_proof.tikz}
		\end{equation}
		where $q'$, $f'$, $g'$, and $h'$ are deterministic. 
		By setting
		\begin{equation}
			\input{aldous_hoover_functional_proof2.tikz}
		\end{equation}
		we obtain the statement because deterministic morphisms (which includes copy morphisms) can move into plates by definition.
	\end{proof}
	
	\begin{remark}[Measure-theoretic corollaries of the Synthetic Aldous-Hoover Theorem]\hfill\label{rem:AH_traditional}
    \begin{enumerate}
        \item 
	        Taking $\cC = \BorelStoch$ in \cref{thm:AH_functional} reproduces the most common version of the Aldous--Hoover Theorem.
	        First, note that the omnirandom state $r \colon I \to [0,1]$ can be chosen to be the Lebesgue measure (\cref{lem:omnirandom}).
		Second, if $(\mathcal{X}_{ij})_{i,j \in \N}$ is an array of real-valued random variables with row-column exchangeable joint distribution, then we conclude that there are independent and uniform $[0,1]$-valued random variables $\alpha$, $(\xi_i)_{i \in \N}$, $(\eta_j)_{j \in \N}$ and $(\zeta_{ij})_{i,j \in \N}$ as well as a measurable function $k \colon [0,1]^4 \to \R$ such that
		\begin{equation}
			\bigl( \mathcal{X}_{ij} \bigr)_{i,j\in \N} \stackrel{d}{=} \bigl( k(\alpha, \xi_i, \eta_j, \zeta_{ij}) \bigr)_{i,j\in \N},
		\end{equation}
		where $\stackrel{d}{=}$ denotes equality in distribution.
		Indeed, \cref{eq:aldous_hoover_functional} is exactly the desired equality in distribution, where the extra variables $\alpha$, $\xi_i$, $\eta_j$ and $\zeta_{ij}$ are the outputs of the  occurring instances of the omnirandom state.

       \item 
	       Besides $\BorelStoch$, we can instantiate \cref{thm:AldousHooverWeak} in other categories, such as $\AnaStoch$ and $\UnivStoch$, since by \cref{ex:assumptions} they both satisfy \cref{ass:de_finetti}. 

	       Despite the full Markov subcategory inclusions
	       \[
		       \BorelStoch \subseteq \AnaStoch \subseteq \UnivStoch,
	       \]
	       the statement for each category does not directly imply the statement for its subcategories, since \cref{thm:AldousHooverWeak} does not guarantee that the extra objects $A$, $B$ and $C$ also live in the subcategory.
	       This is different for the strong form of the Aldous--Hoover Theorem, which we will state and prove as \cref{thm:AldousHoover}: there, instantiating the statement in $\UnivStoch$ gives the strongest measure-theoretic result that we have, since it trivially specializes to full Markov subcategories like $\BorelStoch$ and $\AnaStoch$.
    \end{enumerate}
	As we can see, our approach can be used to give a unified treatment of multiple Aldous-Hoover Theorems. 
	The measure-theoretic details are only relevant for showing that the respective Markov category satisfies \cref{ass:de_finetti}.
	The remainder of each proof is synthetic and common across the different versions.\footnotemark
	\footnotetext{There also exists an Aldous-Hoover Theorem for Radon distributions \cite[Theorem 4.1]{towsner2023AH}, which applies to an even broader class of spaces than the universally measurable ones \cite[Example B.36]{forre2021conditional}.
	We have not yet been able to rederive this result from ours.}
	\end{remark}

	Let us introduce the following notation in order to state the strong form of our Synthetic Aldous--Hoover Theorem more concisely.
	\begin{notation}\label{nota:ah}
		Let $n \in \N$.
		\begin{enumerate}
			\item We introduce the following shorthands:
				\begin{equation*}
					\A\coloneqq \lbrace 1, \dots , n \rbrace, \qquad
					\compl{\A}\coloneqq \N \setminus \A,\qquad
					\complsm{i}\coloneqq \A \setminus \lbrace i \rbrace \text{ for any }i \in \A.
				\end{equation*}
			\item In order to save space in large formulas, we also replace the symbol $\times$ with a simple comma. 
				That is, for $\I, \J \subseteq \N$, we write $X^{\I,\J}$ to denote the Kolmogorov power $X^{\I \times \J}$. 
			\item Additionally, we replace $\lbrace i \rbrace$ by $i$ in the interest of cleaner notation.
				For example, $X^{i,\A}$ stands for $X^{\lbrace i \rbrace,\A}$ and $X^i$ stands for $X^{\{i\}}$ as in the \nameref{sec:intro}.
		\end{enumerate}
	\end{notation}

	As indicated in \cref{fig:notation}, we use these conventions mainly in order to denote subproducts of the Kolmogorov power $X^{\N,\N}$.
	Although all Kolmogorov powers with countably infinite exponent are isomorphic, these subproducts are distinguished also by how $X^{\N,\N}$ projects onto them.
	They pick out particular marginals of a morphism $p \colon I \to X^{\N,\N}$.
	For example, $X^{\A,\A}$ refers to those instances of $X$ which correspond to elements of $\A \times \A \subseteq \N \times \N$.
	On the other hand, $X^{\compl{\A}, \complsm{j}}$ consists of outputs labelled by elements of the set
	\begin{equation}
		\Set*[\big]{ (i,k)  \given  i \in \N \setminus \A \text{ and } k \in \A \setminus \{j\} } \subseteq \N \times \N.
	\end{equation}
	Moreover, the full Kolmogorov power naturally factorizes as
	\begin{equation}\label{eq:array_decomp}
		X^{\N,\N} \cong X^{\compl{\A},\A} \otimes X^{\A,\A} \otimes X^{\A,\compl{\A}} \otimes X^{\compl{\A},\compl{\A}}. 
	\end{equation}
	
	\begin{figure}[t]
		\begin{subfigure}[b]{0.45\textwidth}
			\centering
			\tikzsetfigurename{notations}
			\input{figures/notations.tikz}
		\end{subfigure}
		\hspace*{0.5cm}
		\begin{subfigure}[b]{0.45\textwidth}
			\centering
			\tikzsetfigurename{notation2}
			\input{figures/notation2.tikz}
		\end{subfigure}
		\caption{Graphical description of some subproducts of $X^{\N,\N}$ as per \cref{nota:ah}. 
		The first index in the exponent is the row index, the second index the column index, and $\A$ denotes $\{1, \dots, n\}$.}
		\label{fig:notation}
	\end{figure}

	With this in mind, we can now state the version of the Aldous--Hoover Theorem that we view as the strongest.
	\cref{thm:AldousHooverWeak} indeed follows from it by an application of the universal property of the Kolmogorov power $X^{\N,\N}$, but there is no obvious implication in the opposite direction.
	\begin{theorem}[Synthetic Aldous--Hoover Theorem, strong form]\label{thm:AldousHoover}
		Let $\cC$ be a Markov category satisfying \cref{ass:de_finetti}.
		Then for any row-column exchangeable state $p \colon I \to X^{\N, \N}$, we have the following decomposition
		\begin{equation}\label{eq:aldous_hoover1}
			\input{aldous_hoover1.tikz}
		\end{equation}
		given suitable morphisms $f \colon X^{\N, \N} \to X^{\N}$, $g \colon X^{\N, \N} \to X^{\N}$, and $h \colon X^{\N} \otimes X^{\N, \N} \otimes X^{\N} \to X$.
	\end{theorem}
	
	Here, we leave it understood that comparing the two sides of the equation involves an application of the isomorphism of \cref{eq:array_decomp}.
	Note also that by considering the marginal of \cref{eq:aldous_hoover1} on $X^{\comp{\A},i} \otimes X^{\comp{\A},\comp{\A}}$ we can identify $f$ (and similarly also $g$) as a particular tail conditional, just as in the de Finetti Theorem.

To better understand the use of plate notation, we write out explicitly the case of $n=2$ in \cref{fig:AH_2x2}. 
		\begin{figure}[th]
		\begin{equation*}
			\input{aldous_hoover2x2.tikz}
		\end{equation*}
		\caption{The Aldous--Hoover statement for $n=2$. The only purpose of colors is to aid the reader's eye. This highlights that the output $X^{i,j}$ only depend on the tail of its row $X^{\compl{\A},j}$, the tail of its column $X^{i,\compl{\A}}$, and the tail of the whole infinite matrix $X^{\compl{\A},\compl{\A}}$.}\label{fig:AH_2x2}
		\end{figure}

	Let us now prove \cref{thm:AldousHoover} using several auxiliary results as shown in \cref{fig:strategy_ah}.
	We work under \cref{ass:de_finetti} throughout, although again the existence of conditionals is not necessary for some of our proofs such as the one of \cref{prop:shift_covariance}.

	\tikzexternaldisable
	\newlength{\chunit}
	\setlength{\chunit}{\dimexpr\numexpr 1*\textwidth/12 sp\relax}
	\begin{figure}
		\tikzset{%
			basic/.style = {rectangle, rounded corners, draw=black, minimum width=3\chunit, minimum height=1\chunit, text centered, text width=2.8\chunit, inner sep=2pt},
			1D/.style = {basic,fill=Dandelion!15!white}, 
			2D/.style = {basic,fill=Yellow!13!white}, 
			AH/.style = {basic,fill=LimeGreen!18!white}, 
			implies/.style = {-{Implies},double,double equal sign distance},
		}
		\makebox[\textwidth][c]{
			\begin{tikzpicture}[node distance=0.9\chunit,shorten > = 2pt, shorten < = 3pt]
				\node (dFT) 		[1D] 							{de Finetti Theorem\\ (\sref{Theorem}{thm:definetti})};
				\node (sc)		[1D, right=of dFT]				{Shift covariance (\sref{Corollary}{cor:shift_covariance})};
				
				\node (seiII)		[2D, below=of dFT]		{Single entry independence II (\sref{Lemma}{lem:independence2})};
				\node (seiI)		[2D, left=of seiII]		{Single entry independence I (\sref{Lemma}{lem:independence1})};
				\node (rci)		[2D, right=of seiII]		{Row and column independence (\sref{Lemma}{lem:independence3})};
				
				\node (AHT)		[AH, below=of seiII]			{Aldous--Hoover Theorem (\sref{Theorem}{thm:AldousHoover})};
				
				\draw[implies] 		(dFT.south west) -- (seiI.north east); 
				\draw[implies] 		(dFT.south) -- (seiII.north); 
				\draw[implies] 		(dFT.south east) -- (rci.north west); 
				\draw[implies] 		(sc.south) -- (rci.north); 
				\draw[implies] 		(seiI.south east) -- (AHT.north west); 
				\draw[implies] 		(seiII.south) -- (AHT.north); 
				\draw[implies] 		(rci.south west) -- (AHT.north east); 
			\end{tikzpicture}
		}
		\caption{Overall strategy of the proof of \cref{thm:AldousHoover}.
		The specific conditional independence relations established by shift covariance and the three lemmas in the second row are depicted in \cref{fig:shift_covariance,fig:IndependenceLemmas} respectively.
		Orange statements establish conditional independence relations for 1D sequences of random variables while the yellow and green ones are for 2D arrays.
		The arrows are used to indicate the dependence structure in our proof.}
		\label{fig:strategy_ah}
	\end{figure}
	\tikzexternalenable
	
	\subsection{Shift Covariance}
	\label{sec:shift_covariance}

	To prove our Synthetic Aldous--Hoover Theorem using the steps illustrated in \cref{fig:strategy_ah}, we need to establish three independence lemmas that concern arrays of random variables. 
	We start by working towards the shift covariance property, which is a key ingredient in the proof of the row and column independence lemma.


	This shift covariance is reminiscent of \cref{prop:shift_perm_inv}, one of whose directions establishes that a morphism displaying permutation invariance of its input sequence is necessarily also shift invariant.
	Here, we consider a morphism, whose input and output are both Kolmogorov powers, and we show that its \emph{permutation covariance} implies its \emph{shift covariance}.

	\begin{proposition}[Permutation covariance implies shift covariance]\label{prop:shift_covariance}
		Let $p \colon A \to X^{\N}$ be an exchangeable morphism. 
		If $f \colon X^{\N} \to Y^{\N}$ is a morphism satisfying
		\begin{equation}\label{eq:perm_covariance}
			\input{perm_covariance.tikz}
		\end{equation}
		for every finite permutation $\sigma$, then we also have
		\begin{equation}\label{eq:cond_inv_dex}
			\input{cond_inv_dex.tikz}
		\end{equation}
		where $s \colon \N \to \N$ is the successor function.
	\end{proposition}
	\begin{proof}
		The proof is very similar to that of \cref{prop:shift_perm_inv}.
		
		Using permutation covariance of $f$ together with exchangeability of $p$ and the fact that $X^\sigma$ is deterministic, we find
		\begin{equation}\label{eq:shift_covariance_proof1}
			\input{shift_covariance_proof1.tikz}
		\end{equation}
		so that the morphism on the right is exchangeable when viewed as a morphism of type $A \to (Y \otimes X)^\N$. 	
		Therefore, by \cite[Lemma 5.1]{fritz2021definetti}, it is also invariant under applying the shift $(Y \otimes X)^s$.
		Using this fact together with the shift invariance of $p$ and the fact that $X^s$ is deterministic gives us
		\begin{equation}\label{eq:shift_covariance_proof2}
			\input{shift_covariance_proof2.tikz}
		\end{equation}
		Similarly, let us argue that 
		\begin{equation}\label{eq:shift_covariance_proof3}
			\input{shift_covariance_proof3.tikz}
		\end{equation}
		holds.
		Namely, post-composing \cref{eq:shift_covariance_proof1} with $\id_{Y^\N} \otimes f$ and applying permutation covariance once again shows that the morphism on the left-hand side of \cref{eq:shift_covariance_proof3} is exchangeable as a morphism of type $A \to (Y \otimes Y)^\N$.
		Applying \cite[Lemma 5.1]{fritz2021definetti} again then gives the desired shift invariance.
		
		Combining \cref{eq:shift_covariance_proof2,eq:shift_covariance_proof3} yields
		\begin{equation}\label{eq:shift_covariance_proof4}
			\input{shift_covariance_proof4.tikz}
		\end{equation}
		so that the claimed shift covariance of $f$ follows by the Cauchy--Schwarz axiom.
	\end{proof}

	\begin{figure}[t]
		\centering
		\input{figures/shift_covariance.tikz}
		\caption{Visualization of \cref{cor:shift_covariance}, which expresses the conditional independence relations resulting from shift covariance.
			The upper and lower row represent the sequences $X^{\N}$ and $Y^{\N}$ respectively where each $X$ and each $Y$ is an output of a joint state $q$ on the full array $(X \otimes Y)^\N$.
			The index set $\A$ is the subset of the first $n$ natural numbers (\cref{nota:ah}).
			When conditioned on the gray region ($X^{\compl{\A}}$), the outputs in the red ($X^{\A}$) and blue ($Y^{\compl{\A}}$) regions are independent of each other.}
		\label{fig:shift_covariance}
	\end{figure}
	\begin{corollary}[Shift covariance (see \cref{fig:shift_covariance})]\label{cor:shift_covariance}
		Every exchangeable state $q \colon I \to (X \tensor Y)^{\N}$ displays the conditional independence
		\begin{equation}
			Y^{\compl{\A}} \perp X^{ \A} \mid X^{\compl{\A}}.
		\end{equation}
	\end{corollary}
	\begin{proof}
		Let $p \colon I \to X^{\N}$ be the marginal of $q$ on the first row and let $f \colon X^{\N} \to Y^{\N}$ be a conditional of $q$ given $X^{\N}$.
		Exchangeability of $q$ implies that $p$ is exchangeable and that $f$ is permutation covariant.
		The first is immediate, and the second can be proven by noting that \cref{eq:shift_covariance_proof1} holds again and post-composing it with $\id \otimes X^{\sigma^{-1}}$.
		By \cref{prop:shift_covariance}, $f$ is thus also shift covariant.
		
		Then the marginal of $q$ obtained by discarding $Y^\A$ is
		\begin{equation}\label{eq:shift_covariance_condind_proof}
			\input{shift_covariance_condind_proof.tikz}
		\end{equation}
		where in the first equality we apply the shift covariance $n$ times and the second equality follows from the fact that $(X^s)^n$ discards the first $n$ wires.
		\cref{eq:shift_covariance_condind_proof} shows the desired conditional independence by \cite[Proposition 6.9]{chojacobs2019strings}. 
	\end{proof}

	\subsection{Three Independence Lemmas}\label{sec:3_cond_ind}

		In this section, we prove the three conditional independence relations for row-column exchangeable states forming the middle row of \cref{fig:strategy_ah}.
		The structure of these relations is illustrated in \cref{fig:IndependenceLemmas}.
		The proofs are applications of the semigraphoid properties and the Partition Lemma (\cref{lem:partition}), which both prove new conditional independences from given ones, together with the Synthetic de Finetti Theorem and the shift covariance (\cref{cor:shift_covariance}) as starting points.
		Although this reasoning about conditional independences can be translated completely into string diagrams, which the reader may prefer, our choice allows for a cleaner and more compact presentation.

		In the following, ``Dec'' stands for the decomposition and ``WU'' for the weak union semigraphoid property (\cref{lem:semigraphoid}).
		Throughout, we still work under \cref{ass:de_finetti}, use \cref{nota:ah}, and assume that $p \colon I \to X^{\N,\N}$ is a row-column exchangeable state.
		
		\begin{figure}[t]
			\centering
			\begin{subfigure}[b]{0.42\textwidth}
				\centering
				\tikzsetfigurename{ah_fig_indep_lemma1}
				\input{figures/ah_fig_indep_lemma1.tikz}
				\caption*{\hyperref[lem:independence1]{Single entry independence I}: The entries in a $n \times n$ square (red and blue) are mutually independent when conditioned on the tail of all the $n$ rows, the tail of all the $n$ columns, as well as the tail of the whole array (gray).}
			\end{subfigure}
			\hspace{1cm}
			\begin{subfigure}[b]{0.42\textwidth}
				\centering
				\tikzsetfigurename{ah_fig_indep_lemma2}
				\input{figures/ah_fig_indep_lemma2.tikz}
				\caption*{\hyperref[lem:independence2]{Single entry independence II}: A single output (red) is independent of the other rows and columns (blue) when conditioned on the tail of its row, the tail of its column and the tail of the whole array (gray).}
			\end{subfigure}
			\vspace*{0.5cm}
	
			\begin{subfigure}[b]{0.5\textwidth}
				\centering
				\tikzsetfigurename{ah_fig_indep_lemma3}
				\input{figures/ah_fig_indep_lemma3.tikz}
				\caption*{\hyperref[lem:independence3]{Row and column independence}: All $n$ columns and rows (red and blue) are mutually independent when conditioned on the tail of the array (gray).}
			\end{subfigure}
			\caption{Visual representation of the conditional independence relations proved in the three lemmas in \cref{sec:3_cond_ind}.
			Each of the three big squares represents an array of objects indexed by two positive integers.
			Each of the objects is isomorphic to $X$ and represented by a little square. 
			The white area is always marginalized, while the gray area is always conditioned on.}
			\label{fig:IndependenceLemmas}
		\end{figure}
	
		\begin{lemma}[Single entry independence I]\label{lem:independence1}
			We have
			\begin{equation}
				\bigperp_{i,j\in \A} X^{i,j} \largemid X^{\A, \compl{\A}}, X^{\compl{\A},\A}, X^{\compl{\A}, \compl{\A}}.
			\end{equation}
		\end{lemma}
		\begin{proof}
			Since $p$ is row-column exchangeable, it is in particular exchangeable under all finite permutations of the rows, which are indexed by the first components in the superscripts.
			Thus applying the de Finetti Theorem in the form of \cref{eq:definetti_cond_ind} gives $\bigperp_{i\in \A} X^{i, \N} \mid X^{\compl{\A}, \N}$.
			We can then use the following implications
			\begin{equation}
			\begin{array}{cclr}
				\bigperp_{i\in \A} X^{i, \N} \largemid X^{\compl{\A}, \N}&\overset{\text{\cref{cor:partition}}}{\Longleftrightarrow} & X^{i,\N} \largeperp X^{\complsm{i},\N} \largemid X^{\compl{\A}, \N} &\quad \text{for all }i \in \A \\
				&\overset{\text{WU}}{\Longrightarrow}& X^{i, \A} \largeperp X^{\complsm{i},\N} \largemid X^{\compl{\A},\N}, X^{i, \compl{\A}} & \text{for all }i \in \A \\
				&\overset{\text{WU}}{\Longrightarrow}& X^{i, \A} \largeperp X^{\complsm{i},{\A}} \largemid X^{\compl{\A},\N}, X^{i, \compl{\A}}, X^{\complsm{i},\compl{\A}}& \text{for all }i \in \A\\
				&\overset{\text{\cref{cor:partition}}}{\Longleftrightarrow} & \bigperp_{i \in \A} X^{i,\A} \largemid X^{\A, \compl{\A}}, X^{\compl{\A},\A}, X^{\compl{\A}, \compl{\A}}.&
			\end{array}
			\end{equation}
			which together amount to applying `weak union' for multivariate conditional independence.
			Analogously, exchangeability of the columns of $p$ implies $\bigperp_{j \in \A} X^{\A,j} \largemid X^{\A, \compl{\A}}, X^{\compl{\A},\A}, X^{\compl{\A}, \compl{\A}}$, and so by the Partition Lemma (\cref{lem:partition}) we obtain the desired statement.
		\end{proof}
	
		In the next independence lemma, we consider a single output in $X^{i,j}$ in $X^{\A,\A}$ and marginalize the other outputs in the square $\A \times \A$.
	
		\begin{lemma}[Single entry independence II]\label{lem:independence2}
			We have 
			\begin{equation}\label{eq:independence2}
				X^{i,j} \largeperp X^{\complsm{i}, \compl{\A}}, X^{\compl{\A},\complsm{j}} \largemid X^{\compl{\A},j}, X^{i, \compl{\A}}, X^{\compl{\A}, \compl{\A}}.
			\end{equation}
			for all $i,j \in \A$.
		\end{lemma}
		
		\begin{proof}
			We split the proof into the following two statements, to be proven separately:
			\begin{enumerate}
				\item\label{IndII:i} $X^{i,j} \largeperp X^{\compl{\A}, \complsm{j}} \largemid X^{\compl{\A},j}, X^{i, \compl{\A}}, X^{\compl{\A},\compl{\A}}$
				\item\label{IndII:ii} $X^{i,j} \largeperp X^{\complsm{i}, \compl{\A}} \largemid X^{\compl{\A},j} , X^{i, \compl{\A}},  X^{\compl{\A}, \compl{\A}}, X^{\compl{\A},\complsm{j}}$
			\end{enumerate}
			The claimed conditional independence \eqref{eq:independence2} then follows by the contraction semigraphoid axiom.
			
			To show \ref{IndII:i}, we first consider the marginal of $p$ on
			\begin{equation}
				\bigotimes_{k\in \N} \left(X^{i,k} \otimes X^{\compl{\A},k} \right),
			\end{equation}
			which corresponds to removing all rows with index in $\A \setminus \{i\}$ from the array.
			Since this marginal is invariant under all finite permutations of the column index $k$, we have the conditional independence
			\begin{equation}
				{\bigperp_{k \in {\A}} \left(X^{i,k} , X^{\compl{\A},k} \right)\mid X^{i,\compl{\A}}, X^{\compl{\A},\compl{\A}}}
			\end{equation}
			by the de Finetti Theorem (\cref{eq:definetti_cond_ind}).
			To this we can apply the following implications:
			\begin{equation}
				\begin{array}{ccl}
					\bigperp_{k \in {\A}} \left(X^{i,k} , X^{\compl{\A},k} \right) \largemid X^{i,\compl{\A}}, X^{\compl{\A},\compl{\A}} & \overset{\text{\cref{cor:partition}}}{\Longrightarrow} & X^{i,j}, X^{\compl{\A},j} \largeperp X^{i, \complsm{j} }, X^{\compl{\A},\complsm{j}} \largemid X^{i,\compl{\A}}, X^{\compl{\A},\compl{\A}}\\
					&\overset{\text{Dec}}{\Longrightarrow}& X^{i,j}, X^{\compl{\A},j} \largeperp X^{\compl{\A},\complsm{j}} \largemid X^{i,\compl{\A}}, X^{\compl{\A},\compl{\A}} \\
					&\overset{\text{WU}}{\Longrightarrow}& X^{i,j} \largeperp X^{\compl{\A}, \complsm{j}} \largemid X^{\compl{\A},j}, X^{i, \compl{\A}}, X^{\compl{\A},\compl{\A}}.
				\end{array}
			\end{equation}
			Hence \ref{IndII:i} is proven.
			
			To show \ref{IndII:ii}, we apply the de Finetti Theorem to the rows of the array and obtain ${\bigperp_{k\in \A} X^{k, \N} \mid X^{\compl{\A},\N}}$.
			We now conclude as follows:
			\begin{equation}
			\begin{array}{ccl}
				\bigperp_{k\in \A} X^{k, \N}\mid X^{\compl{\A},\N}& \overset{\text{\cref{cor:partition}}}{\Longrightarrow} & X^{i, \N} \largeperp X^{\complsm{i}, \N} \largemid  X^{\compl{\A},\N} \\
				&\overset{\text{Dec}}{\Longrightarrow}&  X^{i, j}, X^{i, \compl{\A}} \largeperp X^{\complsm{i}, \compl{\A}} \largemid X^{\compl{\A},\N} \\
				&\overset{\text{WU}}{\Longrightarrow} & X^{i,j} \largeperp X^{\complsm{i}, \compl{\A}} \largemid X^{i, \compl{\A}}, X^{\compl{\A},\N},
			\end{array}
			\end{equation}
			where the resulting relation can be identified with the one from \ref{IndII:ii} by splitting the object $X^{\compl{\A},\N}$ into a tensor product of $X^{\compl{\A},j}$, $X^{\compl{\A},\complsm{j}}$ and $X^{\compl{\A}, \compl{\A}}$.
		\end{proof}
		
		The final independence lemma establishes a mutual independence between the tails of the first $n$ rows and the first $n$ columns, when conditioned on the array tail $X^{\compl{\A}, \compl{\A}}$.
		
		\begin{lemma}[Row and column independence]\label{lem:independence3}
			We have 
			\begin{equation}
				\bigperp_{i = 1}^{2n} Y_i \largemid X^{\compl{\A}, \compl{\A}},
			\end{equation}
			where we use the shorthand
			\begin{equation}
					Y_i \coloneqq \begin{cases}  
							X^{i, \compl{\A}}		& \text{if } i \in \{1, \ldots, n\} \\
							X^{\compl{\A},i-n}	& \text{if } i \in \{n+1, \ldots, 2n\}.
						\end{cases}	
			\end{equation}
		\end{lemma}
		
		So the $Y_1, \dots, Y_{2n}$ stand for the tails of the first $n$ rows and first $n$ columns together.
		These are the red and blue regions in the third display of \cref{fig:IndependenceLemmas}.
	
		\begin{proof}
			Splitting the whole array $X^{\N \times \N}$ into a tensor product $X^{\N, \A} \otimes X^{\N, \compl{\A}}$ of the first $n$ columns and the rest, we obtain the conditional independence $X^{\A,\compl{\A}} \perp X^{\compl{\A},\A} \mid X^{\compl{\A},\compl{\A}}$ from shift covariance (\cref{cor:shift_covariance}).
			This relation says that the first $n$ rows (together) are independent from the first $n$ columns (together), when conditioned on the tail of the array.
			Moreover, we also recall that the de Finetti Theorem applied to the rows gives $\bigperp_{i \in \A} X^{i, \N} \mid X^{\compl{\A}, \N}$. 
			
			We now argue as follows, where the two topmost relations are the ones we have just established:
			\begin{center}
				\begin{tikzpicture}
					\tikzstyle{level 1}=[sibling distance=7cm]
					\tikzstyle{level 2}=[sibling distance=7cm]
					\node (1){$X^{i, \compl{\A}} \largeperp X^{\complsm{i}, \compl{\A}}, X^{\compl{\A},\A} \largemid X^{\compl{\A}, \compl{\A}}.$} [grow'=up]
					child {node (2){$X^{i, \compl{\A}} \largeperp X^{\complsm{i},\compl{\A}} \largemid X^{\compl{\A}, \A}, X^{\compl{\A}, \compl{\A}}$}edge from parent[implies-,double equal sign distance,shorten <=3pt,parent anchor=north]
						child {node {$X^{i, \N} \largeperp X^{\complsm{i},\N} \largemid X^{\compl{\A}, \A}, X^{\compl{\A}, \compl{\A}}$} edge from parent[implies-,double equal sign distance,shorten <=3pt]
							node[right,font=\footnotesize] {Dec}} node (A) {}
					}
					child {node (3){$X^{i,\compl{\A}} \largeperp X^{\compl{\A},\A} \largemid X^{\compl{\A},\compl{\A}}$}edge from parent[implies-,double equal sign distance,shorten <=3pt,parent anchor=north]
						child {node{$X^{\A,\compl{\A}} \largeperp X^{\compl{\A},\A} \largemid X^{\compl{\A},\compl{\A}}$}edge from parent[implies-,double equal sign distance,shorten <=3pt] 
						node[right,font=\footnotesize] {Dec}} 
						node (B) {}
					};
					\path (A)--(B) node[midway,font=\footnotesize] {Contraction};
				\end{tikzpicture}
			\end{center}
			Analogously, we can also obtain $X^{\compl{\A},j} \largeperp X^{\compl{\A}, \complsm{j}}, X^{\A,\compl{\A}} \largemid X^{\compl{\A}, \compl{\A}}$ by the same reasoning applied with respect to the columns.
			Since these two conditional independences hold for all $i$ and all $j$, the desired statement follows by \cref{cor:partition}.
		\end{proof}

	\subsection{Proof of the Synthetic Aldous--Hoover Theorem}\label{sec:mainProof}

	\begin{proof}[Proof of \cref{thm:AldousHoover}]
		We apply the three independence lemmas to the row-column exchangeable morphism ${p \colon I \to X^{\N, \N}}$. 
		This leads to the existence of morphisms $f$, $g$, $\widetilde{h}_{k \ell}$ and $h_{ij}$ such that
		\begin{equation*}
			\input{aldous_hoover_proof.tikz}
		\end{equation*}
		where we leave it understood that the wires get suitably permuted in order for the equations to make sense, and:
		\begin{enumerate}
			\item The first step follows from \cref{lem:independence1}.
			\item The second step is a consequence of \cref{lem:independence3}. We can choose $f$ and $g$ independently of the indices $i$ and $j$ thanks to the de Finetti Theorem (\cref{thm:definetti}).
			Moreover, we use the invariance of $p$ under shifts of rows and columns to identify its $X^{\compl{\A}, \compl{\A}}$ marginal with $p$ itself by first marginalizing the first $n$ rows and then marginalizing the first $n$ columns from the remaining array.
		\item The third step uses \cref{lem:independence2}. Indeed, by~\cite[Proposition~6.9]{chojacobs2019strings} the considered conditional independence allows the morphism $\widetilde{h}_{ij}$ to be replaced, up to \as{} equality, with a morphism $h_{ij}$ that depends only on $X^{\compl{\A},j}$, $X^{i, \compl{\A}}$ and $X^{\compl{\A}, \compl{\A}}$.
		\end{enumerate}

		It remains to be shown that the morphisms $h_{ij}$ can be chosen independently of $i$ and $j$. 
		Applying a permutation $\sigma_r = (i \ k)$ to the rows and a permutation $\sigma_c = (j \ \ell)$ to the columns of $p$ gives
		\begin{equation}
			\input{aldous_hoover_proof3.tikz}
		\end{equation}
		by using the equation above (with some outputs marginalized and conditioned on $X^{\A,\A}$) together with the row-column exchangeability of $p$.
		In particular, since this holds for all $i, j, k, \ell \in \A$, setting $h \coloneqq h_{ij}$ for any choice of $i, j \in \A$ leads to the desired statement.
	\end{proof}

\section{The Aldous--Hoover Theorem via Markov Properties}\label{sec:markov_prop}

	In this section, we present another proof of \cref{thm:AldousHoover} by interpreting its corresponding string diagram as a Bayesian network and applying the ordered Markov property{\,---\,}a criterion for when a state can be written as a Bayesian network of a given form.
	For the sake of simplicity, throughout this section we assume that $\cC$ is a strict\footnotemark{} Markov category $\cC$ with conditionals.
	\footnotetext{A Markov category is strict if its underlying monoidal category is strict, which means that $A \otimes (B \otimes C) = (A \otimes B) \otimes C$ holds for all objects $A$, $B$, and $C$ and that both associators and unitors are identities. 
	Every Markov category is equivalent to a strict one such that the equivalence preserves copy morphisms \cite[Theorem 10.17]{fritz2019synthetic}. 
	Therefore, strictness is a convenient assumption we can make without loss of generality.}%
	
	A (causal) Bayesian network is defined via a directed acyclic graph (DAG) that encodes causal relationships. 
	\newterm{Generalized causal models}, as introduced in \cite{fritz2022dseparation, fritz2022free}, extend this framework from DAGs to string diagrams in Markov categories.
	This more general approach natively accommodates hidden variables, input variables, and certain symmetries, neither of which can be encoded when using traditional Bayesian networks.
	We now give a brief summary of the main definitions and results from \cite{fritz2022dseparation, fritz2022free} and refer the reader to these papers for a more detailed discussion.
	
	Every DAG can also be represented as a string diagram.
	For example, a DAG of the form
	\begin{equation}\label{eq:linear_model_dag}
		\begin{tikzpicture}[scale=0.7]

\draw (0,0) circle (1cm) node {$X$};
\draw (4,0) circle (1cm) node {$Y$};
\draw (8,0) circle (1cm) node {$Z$};

\draw[-stealth] (1.2,0) -- (2.8,0);
\draw[-stealth] (5.2,0) -- (6.8,0);

\end{tikzpicture}
	\end{equation}
	reads as a string diagram as
	\begin{equation}\label{eq:linear_model_string}
		\varphi \quad = \quad \input{linear_model_string.tikz}	
	\end{equation}
	Each node in the DAG corresponds to a wire in the string diagram, which is the output of a box whose inputs are determined by the incoming arrows of the node in the DAG. 
	The overall output wires correspond to the observable variables while the others are latent.

	The boxes themselves do not represent any particular Markov kernels or physical mechanisms, but rather serve as placeholders for depicting causal relationships and their labels are formal variables.
	Following \cite{fritz2022dseparation}, we use the term \emph{generalized causal model} to refer to any such string diagram.
	Formally these string diagrams serve as morphisms within a Markov category, specifically a \newterm{free Markov category}~\cite{fritz2022dseparation}.
	
	A morphism $f$ in a Markov category $\cC$ is said to be \newterm{compatible} with a generalized causal model $\varphi$ if there exists a strict Markov functor $F$ from the free Markov category containing $\varphi$ to $\cC$ such that $F(\varphi) = f$ holds, and such that $F$ maps the inputs and outputs of $\varphi$ to those of $f$.
	The idea is that such $F$ assigns an object of $\cC$ to every wire in $\varphi$ and a specific morphism in $\cC$ to every box in $\varphi$. 
	In a traditional interpretation, the object is the space of possible values of a variable associated to the wire and the morphism is a conditional probability that generates the values of its output variables given the values of its input variables.
	The requirement $F(\varphi) = f$ together with matching of inputs and outputs amounts to saying that the morphism $f$ can be decomposed as a string diagram whose form matches $\varphi$ (see \cref{fig:compatibility}).
		
	\begin{figure}
	\begin{center}
	\input{compatibility_visual.tikz}
	\end{center}
	\caption{A morphism $f$ is said to be compatible with a causal model $\varphi$ if there is a strict Markov functor that maps every wire $W$ to an object $W' = F(W)$ in $\cC$ and the boxes $\alpha_i$ to morphisms in $\cC$ such that $f = F(\varphi)$.}
	\label{fig:compatibility}
	\end{figure}		
		
	Assuming that conditionals exist, \newterm{Markov properties}\footnotemark{} provide a practical criterion for deciding whether a given morphism $f$ is compatible with a given (generalized) causal model $\varphi$, by reducing the problem to checking whether the morphism $f$ satisfies (a particular subset of) conditional independence relations derived from the model $\varphi$.
	\footnotetext{Note that Markov properties are not directly named after Markov categories or vice versa, although both terms are in honor of Andrey Markov.}%
	
	For example, a state $I \to X \otimes Y \otimes Z$ is compatible with the causal model depicted in \eqref{eq:linear_model_dag} and \eqref{eq:linear_model_string} if and only if it satisfies the conditional independence
	\begin{equation}
		X \perp Z \mid Y.
	\end{equation}
	This reflects the fact that the information flow from $X$ to $Z$ is obstructed by the box that outputs $Y$. 
	In other words, conditioning on the value of $Y$ removes any correlation between $X$ and $Z$.
		
	There are multiple different Markov properties known (see for example \cite[Section 1]{pearl2009causality} in the language of DAGs) that are necessary and sufficient for compatibility:
	\begin{enumerate}
		\item The \newterm{global Markov property} characterizes \emph{all} conditional independences implied by the causal model
			in terms of a combinatorial criterion known as \emph{$d$-separation}.
			In the language of Markov categories, $d$-separation can be phrased as topological disconnectedness after cutting wires in the string diagram \cite{fritz2022dseparation}.
		\item The \newterm{local Markov property} postulates a smaller set of conditional independences compared to the global Markov property. It asks for one conditional independence relation of a specific type for every box, which intuitively requires the outputs of the box to be independent of ``graphical non-descendants'' given its inputs.
		\item Similarly to the local Markov property, the \newterm{ordered Markov property} asks for one conditional independence relation for every box but of a different type, which is determined by a timing function. The timing function can be interpreted as specifying the order in which outputs are generated, and the Markov property intuitively requires that, given all the inputs of a box, the outputs of the box are independent of all the other outputs generated before them or at the same time.
	\end{enumerate}
	While the first two Markov properties have already been formalized in the language of Markov categories in \cite{fritz2022dseparation}, here we develop the ordered Markov property as another necessary and sufficient criterion for the compatibility with a causal model in \cref{ssec:orderedMarkov}, where we also review the local Markov property.
	
	In \cref{ssec:Markov_AldousHoover}, we then employ the ordered Markov property to provide an alternative proof of the Aldous--Hoover Theorem in the strong form of \cref{thm:AldousHoover}.
	We expect this proof to facilitate a systematic study of representation theorems for hierarchical exchangeability \cite{austin2014hierarchical,lee2022deFinetti,jung2021generalization} in the future.


\subsection{The Ordered Markov Property}
\label{ssec:orderedMarkov}

	
	Here, we introduce the ordered Markov property and show that it is a necessary and sufficient condition for a morphism to be compatible with a causal model.
	We start by introducing basic terminology for generalized causal models.
	
	For a given generalized causal model $\varphi$, we denote its set of wires by $W(\varphi)$ and its set of boxes by $B(\varphi)$. 
	For example, the causal model in \eqref{eq:linear_model_string} has
	\begin{equation}
		W(\varphi) = \{X,Y,Z\} \quad \text{ and } \quad B(\varphi) = \{\alpha, \beta,\gamma\}.
	\end{equation}
	For every box $b \in B(\varphi)$, we also write $\In(b)$ for the set of its input wires and $\Out(b)$ for the set of its output wires. 
	For example, in \eqref{eq:linear_model_string}, we have $\In(\beta) = \{X\}$ and $\Out(\beta) = \{Y\}$.

It is convenient to introduce an order relation, which is similar to a causal order, although it considers both wires and boxes.

\begin{definition}
\label{defn:categoricalSep}
Let $\varphi$ be a generalized causal model.
\begin{enumerate}
	\item\label{it:order_direct}
		For a box $b \in B(\varphi)$ and wire $A \in W(\varphi)$, we write:
		\begin{enumerate}
			\item $A \to b$ if $A \in \In(b)$,
			\item $b \to A$ if $A \in \Out(b)$.
		\end{enumerate}
		In particular, we write $A \to B$ if there is a box $b$ such that $A \to b \to B$ and $b \to c$ if there is a wire $B$ such that $b \to B \to c$. 
	\item The reflexive transitive closure of $\to$ on $B(\varphi) \cup W(\varphi)$ is denoted by $\twoheadrightarrow$.
	\item For a box $b \in B(\varphi)$, its set of \newterm{non-descendants} is 
		\begin{equation*}
			\NonDesc(b) \coloneqq \{X \in W(\varphi) \mid b \not\twoheadrightarrow X\}.
		\end{equation*}
	\end{enumerate}
\end{definition}
For instance, the causal model $\varphi$ in \eqref{eq:linear_model_string} satisfies $\alpha \to X \to \beta \to Y \to \gamma \to Z$, and hence $\alpha \twoheadrightarrow Z$. 
In this example, the set of non-descendants of $\gamma$ is $\lbrace X,Y \rbrace$.

To further illustrate the concept, let us also consider the causal model given by a single box $\alpha$ with outputs $X$ and $Y$, and input $A$, i.e.\
	\begin{equation}
		\input{box_outputXY.tikz}
	\end{equation}
	Then we have $\alpha \rightarrow X$ and $\alpha \rightarrow Y$, while $X \not\twoheadrightarrow Y$ and $Y \not\twoheadrightarrow X$.

	From now on, we restrict our investigation to the following kind of models.\footnote{The results of this subsection, and in particular \cref{thm:causalCompdSep}, can also be developed for the case $\In(\varphi) \neq \emptyset$ in the same fashion as in \cite{fritz2022dseparation}. For this purpose, one needs to introduce a certain asymmetric notion of conditional independence for morphisms (see \cite[Definition 20]{fritz2022dseparation}). Since we apply the present results only to causal models without inputs, we restrict to $\In(\varphi) = \emptyset$ for simplicity.}

\begin{definition}
	\label{def:bloom_without_inputs}
	A generalized causal model $\varphi$ is a \newterm{DAG-like model} if:
	\begin{enumerate}
		\item $\varphi$ is pure bloom~\cite[Definition 10]{fritz2022dseparation}, i.e.\ every wire is an overall output wire in exactly one way.
		\item $\varphi$ has no overall inputs, i.e.~$\In(\varphi) = \emptyset$.
		\item All boxes in $\varphi$ have distinct types in the underlying monoidal signature $\Sigma$.
	\end{enumerate}
\end{definition}

The idea behind our terminology is that writing a DAG as a string diagram results in a DAG-like model. 
On the other hand, generalized causal models that arise in this way are exactly DAG-like for which, additionally, every box has exactly one output~\cite{fritz2022dseparation}.

For example, the causal model in \eqref{eq:linear_model_string} is a Bayesian network, while
\begin{equation}
	\input{non_pure_bloom_morphism.tikz}
\end{equation}
is not because it is not pure bloom.

To state the local Markov property in our setting following~\cite[Definition 33]{fritz2022dseparation}, we recall the notation used there.
If $\varphi$ is a DAG-like model and $p$ is a state in a Markov category $\cC$ for which we are interested in determining compatibility with $\varphi$, then for every wire $X \in W(\varphi)$ we write $X'$ for the corresponding output of $p$, and similarly for a family of outputs.\footnote{An exception is conditional independence relations, where we omit the primes $'$ for simplicity and leave it understood that they refer to $p$ rather than to $\phi$.}
So as described in \cref{fig:compatibility}, compatibility amounts to the existence of a Markov functor $F$ with $F(\varphi) = p$ and $F(X) = X'$ for every $X \in W(\varphi)$.

\begin{definition}[local Markov property]\label{def:MarkovProperties}
	Let $\varphi$ be a DAG-like model and $p \colon I \to \Out(\varphi)'$ a state in $\cC$.
	Then $p$ satisfies the \newterm{local Markov property} with respect to $\varphi$ if for every box $b \in B(\varphi)$, we have
		\begin{equation*}
			\Out(b) \perp \NonDesc(b) \setminus \In(b) \ | \ \In(b) \; \text{ in } p.
		\end{equation*}
\end{definition}

\begin{example}
	For the causal model in \eqref{eq:linear_model_string}, taking $b = \gamma$ produces the only non-trivial conditional independence relation in this case, which is $Z \perp X \ | \ Y$ as mentioned above.
\end{example}

To introduce the ordered Markov property, we need one further concept.

\begin{definition}
A \newterm{timing function} for a generalized causal model $\varphi$ is a map $\timing \colon B(\varphi) \to \N$ such that
\begin{equation}
	\label{eq:timing_def}
	b \to c \quad \Longrightarrow \quad \timing(b) < \timing(c).
\end{equation}
\end{definition}
It follows that if $b \twoheadrightarrow c$, then also $\timing(b) \le \timing(c)$, and in fact $\timing(b) < \timing(c)$ if $b \neq c$.
We think of $\timing(b) < \timing(c)$ as saying that the event $c$ happens after the event $b$, and this motivates \cref{eq:timing_def}, since if $b$ causes $c$, then $c$ should happen after $b$.
While the relevant information encoded in $\timing$ is merely the induced preorder relation on boxes, we find it intuitive to introduce this in the form of a function.
If a timing function is injective, then we obtain a total order on the boxes; for a string diagram obtained from a DAG, this corresponds to the standard ordering of the variables in the context of the ordered Markov property~\cite[Theorem~1.2.6]{pearl2009causality}.

\begin{example}\label{ex:timing_triangle}
	Consider the following DAG-like model:
	\begin{equation}\label{eq:ex_timing}
		\varphi \quad \coloneqq \quad \ \input{ex_timing.tikz}
	\end{equation}
	According to $\varphi$, the boxes $\eta$ and $\mu$ are caused by $\beta$ and $\gamma$ respectively, and these two are caused by $\alpha$.
	One example of a timing function, which is implicitly suggested by how we have drawn \cref{eq:ex_timing}, is given by
	\begin{equation}\label{eq:timing_ex}
		\timing (\alpha)=1, \qquad \timing (\beta)= \timing (\gamma)=2, \qquad \timing (\eta)=\timing(\mu)=3.
	\end{equation}	
\end{example}

\begin{remark}
	Although we could define a timing function alternatively on wires, defining it on boxes is more convenient for the proof of \cref{thm:causalCompdSep} below.	
	This choice is also consistent with the intuitive idea that boxes truly represent events, while wires merely encode connections between events or the flow of information.
\end{remark}

\begin{definition}
	For a generalized causal model $\varphi$ with timing function $\timing$, the \newterm{past} of a box $b \in B(\varphi)$ is
	\begin{equation}
		\Past_{\timing}(b) \coloneqq \bigcup_{c \in B(\varphi) \, : \, \timing(c) \,\le\, \timing(b)} \Out(c).
	\end{equation}
\end{definition}

Every wire in the past of $b$ is a non-descendant of $b$; more explicitly,
\begin{equation}\label{eq:inclusion_dec}
	\Past_{\timing}(b) \setminus \Out(b) \subseteq \NonDesc(b).
\end{equation}
Indeed to show the contrapositive, consider a wire $A \in W(\varphi)$ with $b \twoheadrightarrow A$.
Then $A$ is an output of some box $c$ with $b \twoheadrightarrow c$. 
If $c = b$, then we are done.
Otherwise we get $\timing(b) < \timing(c)$, and hence $A \not\in \Past_{\timing}(b)$.
This proves \cref{eq:inclusion_dec}.

\begin{example}\label{ex:Dec_Cl_containment}
	In general, the inclusion of \cref{eq:inclusion_dec} is strict.
	For example, let $\varphi$ be the causal model defined in \eqref{eq:ex_timing}, and let $\timing$ be the timing function defined in \eqref{eq:timing_ex}.
Then for the box $\beta$, we have
\begin{equation}
	\Past_{\timing}(\beta) \setminus \Out(\beta) = \{A,B,W\} \subsetneq \{ A,B,W,Z \} = \NonDesc(\beta).
\end{equation}
Moreover, there does not exist a timing function $\timing$ for which the equality $\Past_{\timing}(b) \setminus \Out(b) = \NonDesc(b)$ holds for all boxes $b$. 
This can be made precise by noting that whenever $\timing(\beta)\ge \timing (\mu)$, then $\timing (\gamma)< \timing (\eta)$, and in this case $\Past_{\timing}(\gamma) \setminus \Out(\gamma) \subsetneq\NonDesc(\gamma)$.
In conclusion, non-descendants typically take future events into account.
\end{example}

We are now able to define the ordered Markov property, which is quite similar to the local Markov property, but with the past playing the role of the non-descendants.

\begin{definition}[Ordered Markov property]
	Let $\varphi$ satisfy \cref{def:bloom_without_inputs} and let $p \colon I \to \Out(\varphi)'$ be a morphism in $\cC$.
	Then $p$ satisfies the \newterm{ordered Markov property} with respect to $\varphi$ if for every box $b \in B(\varphi)$, we have
		\begin{equation*}
			\Out(b) \perp \Past_{\timing}(b) \setminus (\In(b)\cup \Out(b)) \ | \ \In(b) \; \text{ in } p.
		\end{equation*}
\end{definition}

In words, conditioned on the inputs of a box, its outputs should be independent of all other things happening earlier or at the same time.
If $\timing$ is injective, then this specializes to the classical ordered Markov property for DAGs (see \cite[Theorem 1.2.6]{pearl2009causality}). 
It is worth noting that, based on the discussion at the end of \cref{ex:Dec_Cl_containment}, it is not always possible to retrieve the local Markov property from the ordered one.

We now show the main statement of this section, namely that the ordered Markov property is necessary and sufficient for the compatibility with a generalized causal model.

\begin{theorem}
	\label{thm:causalCompdSep}
	Let $\cC$ be a strict Markov category with conditionals and suppose that:
	\begin{itemize}
		\item $\varphi \colon I \to \bigotimes_{j=1}^m V_j$ is a DAG-like model (\cref{def:bloom_without_inputs}).
		\item $p \colon I \to \bigotimes_{j=1}^m V'_j$ is a state in $\cC$.
	\end{itemize}
	Then the following statements are equivalent:
	\begin{enumerate}
		\item\label{thm_compat} $p$ is compatible with $\varphi$.
		\item\label{thm_local} $p$ satisfies the local Markov property.
		\item\label{thm_ordered} $p$ satisfies the ordered Markov property.
	\end{enumerate}
\end{theorem}

\begin{proof}
	The equivalence \ref{thm_compat} $\Longleftrightarrow$ \ref{thm_local} was already shown in \cite[Theorem 34]{fritz2022dseparation}.
	We refine the argument given there to prove the equivalence with \ref{thm_ordered}.

	\ref{thm_local} $\Longrightarrow$ \ref{thm_ordered}: 
	According to the local Markov property, for every box $b$, we have the conditional independence
\begin{equation}
	\Out(b) \perp \NonDesc(b) \setminus \In(b)  \ | \ \In(b).
\end{equation}
Then \cref{eq:inclusion_dec} and the decomposition property from \cref{lem:semigraphoid}\ref{semi:decomposition} yields the conditional independence
\begin{equation}\label{eq:ord_Markov_proof}
	\Out(b) \perp \Past_{\timing}(b) \setminus (\In(b)\cup \Out(b))  \ | \ \In(b),
\end{equation}
and hence the ordered Markov property.
	
	\ref{thm_ordered} $\Longrightarrow$ \ref{thm_compat}:
This is analogous to the proof of \cite[Theorem 34]{fritz2022dseparation}, but we include the details here for completeness in terms of the same notation.
We prove the claim by induction on the number of boxes $k \coloneqq |B(\varphi)|$.
The case $k = 1$ is trivial, since then compatibility always holds. For the step from $k$ to $k+1$, let $b$ be a box in $\varphi$ of maximal time, i.e.~such that $\timing(b) \ge \timing (c)$ for all $c \in B(\varphi)$. 
Then $\varphi$ in particular factorizes as\footnote{Here the unlabeled wires with dots refer to all the outputs besides the labeled ones.}
\begin{equation*}
\varphi \quad = \quad \input{proofLocalMarkov0.tikz}
\end{equation*}
where $\psi$ is another DAG-like model, and no box in $\psi$ has the same type in $\Sigma$ as $b$ does.
By assumption, $p$ satisfies the ordered Markov property with respect to $b$.
Since $\Past_{\timing}(b) = B(\varphi)$ holds by the maximality of $b$, we can thus decompose $p$ as
\begin{equation}\label{eq:proofLocalMarkov1} 
	\input{proofLocalMarkov1.tikz} \quad = \quad \input{proofLocalMarkov4a.tikz}
\end{equation}
for suitable morphisms $q$ and $h$ in $\cC$.

The restriction $\timing_{|_{B(\varphi) \setminus \lbrace b\rbrace}}$ is a valid timing function on the causal model $\psi$. Moreover, by the decomposition property from \cref{lem:semigraphoid}\ref{semi:decomposition}, the ordered Markov property for $p$ with respect to $\varphi$ implies the ordered Markov property for $q$ with respect to $\psi$, since every box remaining in $\psi$ still has the same inputs and outputs as in $\varphi$.
Therefore, the induction hypothesis implies that $g$ is compatible with $\psi$.

Since the box $b$ appears only once in $\varphi$, we can freely define the action of the strict Markov functor $F$ on $b$ as $F(b) \coloneqq h$, and by the compatibility of $g$ with $\psi$, we can moreover ensure that $F(\psi) = g$. Then, we obtain
\begin{equation*}
\input{proofLocalMarkov3.tikz} \quad = \quad \input{proofLocalMarkov4a.tikz} \quad = \quad \input{proofLocalMarkov4.tikz} \quad = \quad \input{proofLocalMarkov5.tikz}
\end{equation*}
where we use Equation \eqref{eq:proofLocalMarkov1} in the first step and in the last the fact that $F$ is a Markov functor.
\end{proof}

\subsection{An Ordered Markov Property from Row-Column-Exchangeability}
\label{ssec:Markov_AldousHoover}

In the following, we prove that a row-column exchangeable state satisfies the ordered Markov property of the Aldous--Hoover causal model, as depicted in the strong form of the statement (\cref{thm:AldousHoover}).
More precisely, for a fixed $\A = \{1,\ldots,n\}$, this causal model is defined as having wires $\{S_{i,j}\}_{i,j \in \A}$ corresponding to the matrix entries, $\{R_i\}_{i \in \A}$ and $\{C_j\}_{j \in \A}$ corresponding to the row and column tails, and one additional wire $T$ for the overall tail, and displays the following connections:
\begin{equation}
	\label{eq:AH_causal_structure}
	\varphi_{AH} \quad = \quad \input{AH_causal_structure.tikz}
\end{equation}
For example for $n = 2$, this causal model reads as
\begin{equation}\label{eq:AH_causal_structure_2}
	\varphi_{AH} \quad = \quad \input{AH_causal_structure_2.tikz}
\end{equation}
where the colors have no significance beyond visual distinction.
To formulate the ordered Markov property, we consider the natural timing function given by 
\begin{equation}
	\timing(\alpha) = 1, \quad \timing(\beta_i) = \timing(\gamma_j) = 2, \quad \timing(\eta_{i,j}) = 3,
\end{equation}
for all values of the indices $i$ and $j$. 
Using \cref{nota:ah}, the induced ordered Markov properties read as follows.

\begin{lemma}[Ordered Markov conditions for the Aldous--Hoover causal model]
\label{lem:AH_orderedMarkov}
The ordered Markov property based on $\timing$ consists of the following conditional independence relations:
\begin{enumerate}
	\item\label{lem:AH_orderedMarkov_i} $R_i \perp R^{\complsm{i}}, C^{\A} \ | \ T$ for every $i \in \A$.
	\item\label{lem:AH_orderedMarkov_ii} $C_j \perp R^{\A}, C^{\complsm{j}} \ | \ T$ for every $j \in \A$.
	\item\label{lem:AH_orderedMarkov_iii} $S_{i,j} \perp S^{\complsm{i},\complsm{j}}, R^{\complsm{i}}, C^{\complsm{j}} \ | \ R_i, C_j, T$ for every $i,j \in \A$.
\end{enumerate}
\end{lemma}

By the Partition \cref{lem:partition}, Properties \ref{lem:AH_orderedMarkov_i} and \ref{lem:AH_orderedMarkov_ii} demand that all $R_{i}$ and $C_{j}$ must be conditionally independent given $T$. Condition \ref{lem:AH_orderedMarkov_iii} demands that the outputs $S_{i,j}$ can be generated independently of the rest when $R_{i}, C_{j}$ and $T$ are given.

\begin{proposition}\label{prop:exch-orderedMarkovCond}
	Let $\cC$ be a strict Markov category with conditionals.
	If $p \colon I \to X^{\N \times \N}$ in $\cC$ is row-column exchangeable, then it satisfies the ordered Markov property with respect to the $\varphi_{AH}$ from \cref{eq:AH_causal_structure} using the following assignments:
$$ \begin{array}{ccc}
 T' & \coloneqq & X^{\compl{\A}, \compl{A}}, \\
 (R_i)' & \coloneqq & X^{i, \compl{\A}}, \\
 (C_j)' & \coloneqq & X^{\compl{\A}, j}, \\
 (S_{i,j})' & \coloneqq & X^{i,j}.
\end{array}$$
\end{proposition}
\begin{proof}
We have to verify the conditional independence relations of \cref{lem:AH_orderedMarkov} for $p \colon I \to X^{\N,\N}$.
Indeed, \ref{lem:AH_orderedMarkov_i} and \ref{lem:AH_orderedMarkov_ii} follow immediately from \cref{lem:independence3} (and are similar to what was used in its proof). To prove \ref{lem:AH_orderedMarkov_iii},
we translate \cref{lem:independence1,lem:independence2} in our current notation. These respectively read as
\begin{equation}
	S_{i,j} \perp S^{\complsm{i},\complsm{j}} \ | \ R^{\A}, C^{\A}, T \qquad \text{and}\qquad S_{i,j} \perp R^{\complsm{i}}, C^{\complsm{j}} \ | \ R_i, C_j, T,
\end{equation}
for every $i,j \in \A$.
An application of the contraction property from \cref{lem:semigraphoid}\ref{semi:contraction} ensures the desired conditional independence. 
\end{proof}

We are now able to prove the strong form of the Synthetic Aldous--Hoover Theorem.

\begin{proof}[Proof of \cref{thm:AldousHoover}]
	By the strictification theorem for Markov categories~\cite[Theorem~10.17]{fritz2019synthetic}, we can assume without loss of generality that $\cC$ is strict.
	In this case, the compatibility with the Aldous--Hoover causal model \eqref{eq:AH_causal_structure} is now immediate from \cref{thm:causalCompdSep} in combination with \cref{prop:exch-orderedMarkovCond}.

	The fact that the morphisms can be chosen independently of $i$ and $j$ is analogous to the argument in the proof of \cref{thm:AldousHoover} we gave in \cref{sec:mainProof}.
\end{proof}

\begin{remark}[Relation between the two proofs]
Our two proofs of the Aldous--Hoover Theorem use similar ingredients, since both rely on the three independence lemmas developed in \cref{sec:3_cond_ind}. However, there still are differences, which impact the proofs' clarity and expected generalizability:
\begin{itemize}
	\item The proof in \cref{sec:mainProof} is \emph{more direct} than the one of this section, by leveraging the three independence lemmas directly to construct the resulting Aldous--Hoover string diagram rather than proceeding through an additional step like \cref{prop:exch-orderedMarkovCond}.
	\item Conversely, the proof of this section adopts a \emph{more systematic framework}, which we expect to be more easily generalizable, such as to representation theorems for hierarchical exchangeability.
\end{itemize}
One might argue that the ordered Markov property, as expressed in \cref{lem:AH_orderedMarkov}, could have been established explicitly in \cref{sec:aldous_hoover}, instead of the three independence lemmas in \cref{sec:3_cond_ind}. 
However, this approach would complicate and lengthen the proof of the Aldous--Hoover Theorem, as it would replicate the inductive strategy used in the proof of \cref{thm:causalCompdSep}.
\end{remark}

\newpage
\appendix

\section{Parametric Markov Categories}\label{sec:param}
	
	One powerful aspect of the category-theoretic approach to probability offered by Markov categories is the fact that many results about states immediately generalize to morphisms with arbitrary domain.
	The Synthetic de Finetti Theorem proven in \cite{fritz2021definetti} is a prime example of this phenomenon.
	One type of arguments that facilitates such generalizations is to instantiate the result for states in \newterm{parametric Markov categories} introduced in \cite{fritz2023representable} and proving that the relevant assumptions automatically carry over to these categories.
	
	For a Markov category $\cC$ and an object $W \in \cC$, the corresponding parametric Markov category $\cC_W$ is defined by taking the same objects as those of $\cC$ and using the hom-sets $\param{\cC_W (A,X)} \coloneqq \cC(A \otimes W, X)$.\footnotemark{}
	\footnotetext{We use the convention $A \otimes W$ instead of $W \otimes A$ as in \cite{fritz2023representable}, which makes no difference since Markov categories are \emph{symmetric} monoidal.}%
	The composition and the monoidal product distribute the parameters $W$ via copying, while the copy and the delete in $\cC_W$ discard the additional parameter $W$.
	We refer to \cite{fritz2023representable} for full details.
	
	\begin{lemma}\label{lem:param_cauchy}
		If $\cC$ satisfies the Cauchy--Schwarz axiom, then for any object $W \in \cC$, the parametric Markov \mbox{category $\cC_W$} satisfies it as well.
	\end{lemma}
		
	\begin{proof}
		Let $\param{q \colon A \to X}$, $\param{h_1\colon X \to Y}$ and $\param{h_2 \colon X \to Z}$ be arbitrary morphisms in $\cC_W$.
		As in \cite{fritz2023representable}, the blue coloring reminds us that these are morphisms in $\cC_W$, which by definition are morphisms in $\cC$ with an extra input $W$. 
		By definition of the composition in $\cC_W$, we have
		\begin{equation}
			\input{cs_parametric.tikz}
		\end{equation}
		Therefore, the antecedent of the Cauchy--Schwarz axiom in $\cC_W$ is translated into the antecedent of the Cauchy--Schwarz axiom in $\cC$ by replacing $q$ with the dashed box. 
		Upon marginalization of the extra $W$ output that we obtain by applying Implication \eqref{eq:spat_causal}, the consequent of the Cauchy--Schwarz axiom in $\cC_W$ follows.
	\end{proof}
	
	It has been shown earlier in \cite[Lemma 2.10]{fritz2023representable} and \cite[Lemma 5.5]{fritz2021definetti} that conditionals and Kolmogorov products carry over from $\cC$ to $\cC_W$.
	Therefore, the following holds. 
	\begin{corollary}\label{cor:param_assumptions}
		Whenever a Markov category $\cC$ satisfies \cref{ass:de_finetti}, then also the parametric Markov category $\cC_W$ satisfies \cref{ass:de_finetti}.
	\end{corollary}
	
	This provides a pathway for extending our results for states from the main text to theorems about general morphisms with non-trivial inputs.

\section{Further de Finetti Theorems}\label{sec:dF_further}

A de Finetti Theorem for morphisms with inputs has been proven in \cite{fritz2021definetti}.
Similarly, by using parametric Markov categories, it is straightforward to give such a version of the Synthetic de Finetti Theorem in the form of \cref{cor:definetti}.

\begin{theorem}[Synthetic de Finetti Theorem, parametric weak form]\label{thm:definetti_W}
	Let $\cC$ be a Markov category satisfying \cref{ass:de_finetti}.
	Then for every exchangeable morphism $p \colon W \to X^{\N}$, there exist an object $A \in \cC$ together with morphisms $q \colon W \to A$ and $f \colon A \to X$ such that we have
	\begin{equation}
		\input{de_finetti_standardplate_W.tikz}
	\end{equation}
\end{theorem}

\begin{proof}
	By \cref{cor:param_assumptions}, the Synthetic de Finetti Theorem (\cref{cor:definetti}) holds in $\cC_W$ with some object $\param{A'} \in \cC_W$. 
	Expressing this result in terms of morphisms in $\cC$ shows that there are morphisms $q' \colon W \to A'$ and $f \colon A' \otimes W \to X$ such that
	\begin{equation}\label{eq:de_finetti_W_proof}
		\input{de_finetti_W_proof.tikz}
	\end{equation}
	The statement now follows upon taking $A \coloneqq A' \otimes W$ and defining $q$ to be the dashed box.
\end{proof}


The generality of \cref{prop:shift_perm_inv} also allows us to prove a variation on the de Finetti Theorem---previously not possible---which keeps track of correlations with other variables.
	This amounts to considering a \emph{dilation} of an exchangeable morphism $p \colon I \to X^\N$, which is a morphism $\pi \colon I \to X^\N \otimes Y$ marginalizing to $p$~\cite{fritz2022dilations}.
	If such $\pi$ is \newterm{exchangeable in the $X$ outputs},\footnote{By this we mean $(X^{\sigma} \otimes \id_Y)\comp \pi = \pi$ for all finite permutations $\sigma$.} which is a stronger condition than just the exchangeability of its marginal $p$, then again we obtain a de Finetti decomposition.
	
	\begin{theorem}[Synthetic de Finetti Theorem for dilations, strong form]
		\label{thm:dF_dilation}
		Let $\cC$ be a Markov category satisfying \cref{ass:de_finetti}. 
		Then every state $\pi \colon I \to X^{\N} \otimes Y$ exchangeable in the $X$ outputs 
		can be written in the form
		\begin{equation}\label{eq:exc_cond_iid_side}
			\input{exc_cond_iid_SideInfo.tikz}
		\end{equation}
		for every $n \in \N$.
	\end{theorem}
	\begin{proof}
		Let $h \colon X^{\N} \to Y$ denote a conditional of $\pi$ given $X^\N$, which by definition means 
		\begin{equation}\label{eq:cond_SideInfo}
			\input{cond_SideInfo.tikz}
		\end{equation}
		Since $h$ is $\as{p}$ permutation invariant (i.e.~$h \ase{p} h \comp X^{\sigma}$ for every finite permutation $\sigma$), it is also $\as{p}$ shift invariant by \cref{prop:shift_perm_inv}.
		Using the shift invariance $n$ times and decomposing $p$ according to \cref{eq:exc_cond_iid} gives
		\begin{equation}\label{eq:exc_cond_iid_SideInfo_proof}
			\input{exc_cond_iid_SideInfo_proof.tikz}
		\end{equation}
		Combining \cref{eq:cond_SideInfo,eq:exc_cond_iid_SideInfo_proof} then yields the desired statement.
	\end{proof}	
		
	We call a morphism $p \colon A \to X^\N$ \newterm{\as{$m$} exchangeable}, where $m \colon I \to A$ is a state, if exchangeability in the form of \cref{eq:exchangeable_def} holds up to $\ase{m}$. 

	\begin{corollary}[Synthetic de Finetti Theorem, strong almost sure version]
		\label{thm:dF_morphism}
		Let $\cC$ be a Markov category satisfying \cref{ass:de_finetti} and $m \colon I \to A$ any state.
		Then for every \as{$m$} exchangeable morphism $p \colon A \to X^{\N}$, we have
		\begin{equation}\label{eq:exc_cond_iid_2}
				\input{exc_cond_iid_2.tikz}
		\end{equation}
		for every $n \in \N$, where $q \coloneqq p \comp m$.
	\end{corollary}
	\begin{proof}
		This statement follows immediately upon applying \cref{thm:dF_dilation} to the morphism
		\begin{equation}
			\input{f_as.tikz}
		\end{equation}
		Indeed this morphism is exchangeable in the $X$ outputs because $p$ is $\as{m}$ exchangeable.
	\end{proof}	
%
	
%
%

\section{Further Aldous--Hoover Theorems}
\label{sec:AH_further}

	We now consider versions of the Synthetic Aldous--Hoover Theorem which are analogous to those of the de Finetti Theorem in the previous section.
	We believe that these results are new even for $\BorelStoch$.

	As before, the weak form of the theorem is amenable to generalization to morphisms with additional inputs via the parametrization construction.
	On the other hand, the strong form require a consideration of invariance under permutations of outputs almost surely.

	
	\begin{theorem}[Synthetic Aldous--Hoover Theorem, parametric weak form]\label{thm:AldousHooverParametricWeak}
		Let $\cC$ be a Markov category satisfying \cref{ass:de_finetti}.
		Then for every row-column exchangeable morphism $p \colon W \to X^{\N \times \N}$, there exist objects $A, B, C \in \cC$ and morphisms
		\begin{equation}
			q \colon W \to A, \qquad f \colon A\to B, \qquad g \colon A \to C, \qquad h \colon B\otimes A \otimes C \longrightarrow X
		\end{equation}
		such that
		\begin{equation}
			\input{aldous_hoover_standardplate_W.tikz}
		\end{equation}
	\end{theorem}
	The proof is analogous to the one of \cref{thm:definetti_W}, so we limit ourselves to a sketch.
	\begin{proof}
	
		By \cref{cor:param_assumptions}, \cref{thm:AldousHooverWeak} holds in the parametric Markov category $\cC_W$ with suitable objects $\param{A'},\param{B'},\param{C'} \in \cC_W$. 
		The statement is now a consequence of translating the result back to $\cC$ and choosing $A\coloneqq A' \otimes W$, $B\coloneqq B' \otimes W$ and $C \coloneqq C' \otimes W$. 
	\end{proof}

	Next, we prove a version of the strong form of our Synthetic Aldous--Hoover Theorem for dilations, which is analogous to \cref{thm:dF_dilation}.
	
	\begin{theorem}[Synthetic Aldous--Hoover theorem for dilations, strong form]
	Let $\cC$ be a Markov category satisfying \cref{ass:de_finetti}. 
	Then every state $\pi \colon I \to X^{\N \times \N} \otimes Y$ which is row-column exchangeable in its $X$ outputs can be written in the form
	\begin{equation}
		\input{aldous_hoover_dilation.tikz}	
	\end{equation}	
	for every $n \in \N$.
	\end{theorem}
	\begin{proof}
		The proof is analogous to the proof of \cref{thm:dF_dilation}.
		We merely need to extend the argument to two dimensions.
		To give a sketch, let $\pi_{|X^{\N,\N}} \colon X^{\N, \N} \to Y$ be a conditional of $\pi$ given $X^{\N, \N}$ and $p \coloneqq  (\id \otimes \discard_Y) \circ \pi$.
		Then $\pi_{|X^{\N,\N}}$ is $\as{p}$ permutation invariant with respect to both axes.
		Thus by applying \cref{prop:shift_perm_inv} to each axis, we conclude that it is also $\as{p}$ shift invariant in each axis.
		Using this shift invariance $n$ times for each axis and decomposing $p$ according to \cref{eq:aldous_hoover1} establishes the statement.
	\end{proof}
	
	For $m \colon I \to A$ a state, let us call a morphism $p \colon A \to X^{\N\times \N}$ \newterm{\as{$\bm{m}$} row-column exchangeable} if row-column exchangeability in the form of \cref{eq:rce_condition} holds up to $\as{m}$ equality.
		
	\begin{corollary}[Synthetic Aldous--Hoover Theorem, strong almost sure version]
		\label{cor:aldous_hoover_morphism}
		Let $\cC$ be a Markov category satisfying \cref{ass:de_finetti} and $m \colon I \to A$ any state.
		Then for every \as{$m$} row-column exchangeable morphism $p \colon A \to X^{\N \times \N}$, we have
		\begin{equation}\label{eq:almost_sure_aldous_hoover}
				\input{aldous_hoover_almost_sure.tikz}
		\end{equation}
	\end{corollary}
	\begin{proof}
	Identical to the proof of \cref{thm:dF_morphism}.
	\end{proof}		

\section{Omnirandomness}\label{sec:omnirandom}

Here, we introduce a strengthening of the notion of randomness pushback~\cite[Definition~11.19]{fritz2019synthetic}.
We use this in the main text in order to derive the Aldous--Hoover Theorem in its most commonly known form.

\begin{definition}
	\label{def:omnirandom}
	A state $r \colon I \to R$ in a Markov category is \newterm{omnirandom} if for every morphism $f \colon X \to Y$, there is a deterministic morphism $g \colon R \otimes X \to Y$ such that
	\begin{equation}
		\label{eq:omnirandom}
		\input{omnirandom.tikz}
	\end{equation}
\end{definition}

The following instance of this is essentially the \emph{noise outsorcing lemma}, also known as \emph{transfer}~\cite[Theorem~5.10]{kallenberg1997foundations}.
It is usually stated for joint distributions rather than Markov kernels, but we find the kernel version to be more intuitive.

\begin{lemma}
	\label{lem:omnirandom}
	In $\BorelStoch$, the Lebesgue measure $\lambda \colon I \to [0,1]$ is omnirandom.
\end{lemma}

We provide a proof for the sake of completeness.
\begin{proof}
	For a given $f \colon X \to Y$, we can assume without loss of generality that its codomain is a closed subset $Y \subseteq \R$ by Kuratowski's theorem.\footnote{Since $Y$ is either finite or isomorphic to $\N$, in which case the statement is clear, or is isomorphic to $\R$ itself.}
	Then define a  measurable function $g \colon [0,1] \otimes X \to \R$ by taking the quantile function,
	\begin{equation}
		g(r,x) \coloneqq \inf \Set{ s  \given  f \bigl( (-\infty,s] \,|\, x \bigr) \geq r }.
	\end{equation}
	We show that $g$ takes values in $Y$. Indeed if $g(r, x) \in \R \setminus Y$, then the facts that $\R \setminus Y$ is open and $f(\ph \,|\, x)$ is supported on $Y$ imply that there is $\eps > 0$ such that $f((-\infty, g(r,x) - \eps] \,|\, x) = f((-\infty, g(r,x)] \,|\, x)$.
	This contradicts the minimality of $g(r,x)$, and hence $g(r,x) \in Y$.
	Note that $g$ is monotone in $r$.

	To verify the desired equation~\eqref{eq:omnirandom}, it is enough to do so on a fixed input $x \in X$ and on a measurable set of the form $(-\infty,s]$.
	Then the right-hand side evaluates to the probability that $g(-,x)$ is at most $s$, which by monotonicity of $g$ is
	\begin{equation}
		\sup \{ r \mid g(r,x) \le s \} = \sup \{ r \mid f((-\infty,s] \,|\, x) \ge r \} = f((-\infty,s] \,|\, x),
	\end{equation}
	as was to be shown.
\end{proof}

\bibliographystyle{plain}
\bibliography{../markov}

\end{document}